\documentclass[10pt,twoside]{article}
\usepackage{amsmath,amsthm, amssymb, amsfonts, amscd}
\usepackage{verbatim}
\usepackage[all]{xy}
\usepackage{amsxtra}
\usepackage{url}
\usepackage{enumerate}
\def\YEAR{\year}\newcount\VOL\VOL=\YEAR\advance\VOL by-1995
\def\firstpage{1}\def\lastpage{1000}
\def\received{}\def\revised{}
\def\communicated{}

\makeatletter
\def\magnification{\afterassignment\m@g\count@}
\def\m@g{\mag=\count@\hsize6.5truein\vsize8.9truein\dimen\footins8truein}
\makeatother

\font\eightrm=cmr8
\font\caps=cmcsc10                    
\font\Caps=cmcsc10 scaled \magstep1   

\makeatletter
\@specialpagefalse\headheight=8.5pt
\def\DocMath{}
\renewcommand{\@evenhead}{%
    \ifnum\thepage>\lastpage\rlap{\thepage}\hfill%
    \else\rlap{\thepage}\slshape\leftmark\hfill{\caps\SAuthor}\hfill\fi}%
\renewcommand{\@oddhead}{%
    \ifnum\thepage=\firstpage{\DocMath\hfill\llap{\thepage}}%
    \else{\slshape\rightmark}\hfill{\caps\STitle}\hfill\llap{\thepage}\fi}%
\makeatother

\def\TSkip{\bigskip}
\newbox\TheTitle{\obeylines\gdef\GetTitle #1
\ShortTitle  #2
\SubTitle    #3
\Author      #4
\ShortAuthor #5
\EndTitle
{\setbox\TheTitle=\vbox{\baselineskip=20pt\let\par=\cr\obeylines%
\halign{\centerline{\Caps##}\cr\noalign{\medskip}\cr#1\cr}}%
    \copy\TheTitle\TSkip\TSkip%
\def\next{#2}\ifx\next\empty\gdef\STitle{#1}\else\gdef\STitle{#2}\fi%
\def\next{#3}\ifx\next\empty%
    \else\setbox\TheTitle=\vbox{\baselineskip=20pt\let\par=\cr\obeylines%
    \halign{\centerline{\caps##} #3\cr}}\copy\TheTitle\TSkip\TSkip\fi%
\centerline{\caps #4}\TSkip\TSkip%
\def\next{#5}\ifx\next\empty\gdef\SAuthor{#4}\else\gdef\SAuthor{#5}\fi%
\ifx\received\empty\relax
    \else\centerline{\eightrm Received: \received}\fi%
\ifx\revised\empty\TSkip%
    \else\centerline{\eightrm Revised: \revised}\TSkip\fi%
\ifx\communicated\empty\relax
    \else\centerline{\eightrm Communicated by \communicated}\fi\TSkip\TSkip%
\catcode'015=5}}\def\Title{\obeylines\GetTitle}
\def\Abstract{\begingroup\narrower
    \parskip=\medskipamount\parindent=0pt{\caps Abstract. }}
\def\EndAbstract{\par\endgroup\TSkip}

\long\def\MSC#1\EndMSC{\def\arg{#1}\ifx\arg\empty\relax\else
     {\par\narrower\noindent%
     2010 Mathematics Subject Classification: #1\par}\fi}

\long\def\KEY#1\EndKEY{\def\arg{#1}\ifx\arg\empty\relax\else
    {\par\narrower\noindent Keywords and Phrases: #1\par}\fi\TSkip}

\newbox\TheAdd\def\Addresses{\vfill\copy\TheAdd\vfill
    \ifodd\number\lastpage\vfill\eject\phantom{.}\vfill\eject\fi}
{\obeylines\gdef\GetAddress #1
\Address #2
\Address #3
\Address #4
\EndAddress
{\def\xs{4.3truecm}\parindent=0pt
\setbox0=\vtop{{\obeylines\hsize=\xs#1\par}}\def\next{#2}
\ifx\next\empty 
     \setbox\TheAdd=\hbox to\hsize{\hfill\copy0\hfill}
\else\setbox1=\vtop{{\obeylines\hsize=\xs#2\par}}\def\next{#3}
\ifx\next\empty 
     \setbox\TheAdd=\hbox to\hsize{\hfill\copy0\hfill\copy1\hfill}
\else\setbox2=\vtop{{\obeylines\hsize=\xs#3\par}}\def\next{#4}
\ifx\next\empty\ 
     \setbox\TheAdd=\vtop{\hbox to\hsize{\hfill\copy0\hfill\copy1\hfill}
                \vskip20pt\hbox to\hsize{\hfill\copy2\hfill}}
\else\setbox3=\vtop{{\obeylines\hsize=\xs#4\par}}
     \setbox\TheAdd=\vtop{\hbox to\hsize{\hfill\copy0\hfill\copy1\hfill}
            \vskip20pt\hbox to\hsize{\hfill\copy2\hfill\copy3\hfill}}
\fi\fi\fi\catcode'015=5}}\gdef\Address{\obeylines\GetAddress}

\DeclareFontEncoding{OT2}{}{} 
\DeclareTextFontCommand{\textcyr}{\fontencoding{OT2}
    \fontfamily{wncyr}\fontseries{m}\fontshape{n}\selectfont}
\newcommand{\Sha}{\textcyr{Sh}}

\begin{document}
\Title
Stably Cayley semisimple groups
\ShortTitle
\SubTitle
\Author
Mikhail Borovoi and Boris Kunyavski\u\i
\ShortAuthor
M. Borovoi and B. Kunyavski\u\i
\EndTitle
\Abstract
A linear algebraic group $G$  over a field $k$
is called a Cayley group if it admits a Cayley map,
i.e., a $G$-equivariant birational isomorphism over $k$
between the group variety $G$ and its Lie algebra ${\rm Lie}(G)$.
A prototypical example is the classical
``Cayley transform" for the special orthogonal group ${\bf SO}_n$  defined
by Arthur Cayley in 1846.
A linear algebraic group $G$ is called stably Cayley
if $G \times S$ is Cayley for some split $k$-torus $S$.
We classify stably Cayley semisimple groups
over an arbitrary field $k$ of characteristic 0.
\EndAbstract

\MSC
20G15, 20C10.
\EndMSC
\KEY
Linear algebraic group, stably Cayley group,
quasi-permutation lattice.
\EndKEY
\Address
Mikhail Borovoi
Raymond and Beverly
Sackler School of
Mathematical Sciences,
Tel Aviv University
6997801 Tel Aviv
Israel
borovoi@post.tau.ac.il
\Address
Boris Kunyavski\u\i
 Department of Mathematics
Bar-Ilan University
5290002 Ramat Gan
 Israel
kunyav@macs.biu.ac.il
\Address
\Address
\EndAddress 

\newtheorem{lemma}{Lemma}[section]
\newtheorem{proposition}[lemma]{Proposition}
\newtheorem{theorem}[lemma]{Theorem}
\newtheorem{corollary}[lemma]{Corollary}
\newtheorem{condition}[lemma]{Condition}
\newtheorem{exercise}[lemma]{Exercise}

\newtheorem{lemmas}{Lemma}[subsection]
\newtheorem{propositions}[lemmas]{Proposition}
\newtheorem{theorems}[lemmas]{Theorem}
\newtheorem{corollarys}[lemmas]{Corollary}
\newtheorem{conditions}[lemmas]{Condition}

\theoremstyle{definition}

\newtheorem{example}[lemma]{Example}
\newtheorem{definition}[lemma]{\sl Definition}
\newtheorem{conjecture}[lemma]{Conjecture}
\newtheorem{question}[lemma]{Question}

\newtheorem{examples}[lemmas]{Example}
\newtheorem{definitions}[lemmas]{Definition}
\newtheorem{conjectures}[lemmas]{Conjecture}
\newtheorem{questions}[lemmas]{Question}

\newtheorem{subsec}[lemma]{}

\theoremstyle{remark}

\newtheorem{remark}[lemma]{Remark}
\newtheorem{remarks}[lemmas]{Remark}
\newtheorem{notation}{Notation}\def\thenotation{}

\numberwithin{equation}{section}


\newcommand{\m}{^{\times}}

\newcommand{\Gspl}{{G}_{\rm spl}}
\newcommand{\Tspl}{{T}_{\rm spl}}
\newcommand{\Bspl}{{B}_{\rm spl}}

\newcommand{\ff}{F^{\times}}
\newcommand{\fs}{F^{\times 2}}
\newcommand{\llg}{\longrightarrow}
\newcommand{\tens}{\otimes}
\newcommand{\inv}{^{-1}}
\newcommand{\Dfn}{\stackrel{\mathrm{def}}{=}}
\newcommand{\iso}{\stackrel{\sim}{\to}}
\newcommand{\leftexp}[2]{{\vphantom{#2}}^{#1}{#2}}
\newcommand{\mult}{\mathrm{mult}}
\newcommand{\sep}{\mathrm{sep}}
\newcommand{\id}{\mathrm{id}}
\newcommand{\diag}{\mathrm{diag}}
\newcommand{\op}{^{\mathrm{op}}}
\newcommand{\ra}{\rightarrow}
\newcommand{\xra}{\xrightarrow}
\newcommand{\gen}{\mathrm{gen}}
\newcommand{\CH}{\operatorname{CH}}
\newcommand{\Span}{\operatorname{Span}}
\renewcommand{\Im}{\operatorname{Im}}
\newcommand{\Ad}{\operatorname{Ad}}
\newcommand{\NN}{\operatorname{N}}
\newcommand{\Ker}{\operatorname{Ker}}
\newcommand{\Pic}{\operatorname{Pic}}
\newcommand{\Tor}{\operatorname{Tor}}
\newcommand{\Lie}{\operatorname{Lie}}
\newcommand{\ind}{\operatorname{ind}}
\newcommand{\ch}{\operatorname{char}}
\newcommand{\Inv}{\operatorname{Inv}}
\newcommand{\Int}{\operatorname{Int}}
\newcommand{\Inn}{\operatorname{Inn}}
\newcommand{\SInn}{\operatorname{SInn}}
\newcommand{\res}{\operatorname{res}}
\newcommand{\cor}{\operatorname{cor}}
\newcommand{\Br}{\operatorname{Br}}
\newcommand{\Nil}{\operatorname{Nil}}
\newcommand{\Spec}{\operatorname{Spec}}
\newcommand{\Proj}{\operatorname{Proj}}
\newcommand{\SK}{\operatorname{SK}}
\newcommand{\Gal}{\operatorname{Gal}}
\newcommand{\SL}{\operatorname{SL}}
\newcommand{\PGL}{\operatorname{PGL}}
\newcommand{\GL}{\operatorname{GL}}
\newcommand{\gSL}{\operatorname{\mathbf{SL}}}
\newcommand{\gO}{\operatorname{\mathbf{O}}}
\newcommand{\gSO}{\operatorname{\mathbf{SO}}}
\newcommand{\gPSO}{\operatorname{\mathbf{PSO}}}
\newcommand{\gG}{\operatorname{\mathbf{G}}}
\newcommand{\gSp}{\operatorname{\mathbf{Sp}}}
\newcommand{\Sympl}{\operatorname{\gSp}}
\newcommand{\gGL}{\operatorname{\mathbf{GL}}}
\newcommand{\gPGL}{\operatorname{\mathbf{PGL}}}
\newcommand{\gPGU}{\operatorname{\mathbf{PGU}}}
\newcommand{\gSpin}{\operatorname{\mathbf{Spin}}}
\newcommand{\gSU}{\operatorname{\mathbf{SU}}}
\newcommand{\gPSU}{\operatorname{\mathbf{PSU}}}
\newcommand{\gU}{\operatorname{\mathbf{U}}}
\newcommand{\End}{\operatorname{End}}
\newcommand{\Hom}{\operatorname{Hom}}
\newcommand{\Mor}{\operatorname{Mor}}
\newcommand{\Map}{\operatorname{Map}}
\newcommand{\Aut}{\operatorname{Aut}}
\newcommand{\Coker}{\operatorname{Coker}}
\newcommand{\Ext}{\operatorname{Ext}}
\newcommand{\Nrd}{\operatorname{Nrd}}
\newcommand{\Norm}{\operatorname{Norm}}
\newcommand{\spann}{\operatorname{span}}
\newcommand{\Symd}{\operatorname{Symd}}
\newcommand{\Sym}{\operatorname{S}}
\newcommand{\red}{\operatorname{red}}
\newcommand{\Prp}{\operatorname{Prp}}
\newcommand{\Prd}{\operatorname{Prd}}
\newcommand{\tors}{\operatorname{tors}}
\newcommand{\Tr}{\operatorname{Tr}}
\newcommand{\Trd}{\operatorname{Trd}}
\newcommand{\disc}{\operatorname{disc}}
\newcommand{\divi}{\operatorname{div}}
\newcommand{\GCD}{\operatorname{g.c.d.}}
\newcommand{\rank}{\operatorname{rank}}
\newcommand{\A}{\mathbb{A}}
\renewcommand{\P}{\mathbb{P}}
\newcommand{\Z}{\mathbb{Z}}
\newcommand{\bbZ}{\mathbb{Z}}
\newcommand{\bbG}{\mathbb{G}}
\newcommand{\bbQ}{\mathbb{Q}}
\newcommand{\N}{\mathbb{N}}
\newcommand{\F}{\mathbb{F}}
\newcommand{\Q}{\mathbb{Q}}
\newcommand{\R}{\mathbb{R}}
\newcommand{\C}{\mathbb{C}}
\newcommand{\QZ}{\mathop{\mathbb{Q}/\mathbb{Z}}}
\newcommand{\gm}{\mathbb{G}_m}
\newcommand{\hh}{\mathbb{H}}

\newcommand{\cA}{\mathcal A}
\newcommand{\cB}{\mathcal B}
\newcommand{\cC}{\mathcal C}
\newcommand{\cU}{\mathcal U}
\newcommand{\cI}{\mathcal I}
\newcommand{\cJ}{\mathcal J}
\newcommand{\cO}{\mathcal O}
\newcommand{\cF}{\mathcal F}
\newcommand{\cG}{\mathcal G}
\newcommand{\cL}{\mathcal L}
\newcommand{\cP}{\mathcal P}

\newcommand{\falg}{F\mbox{-}\mathfrak{alg}}
\newcommand{\fgroups}{F\mbox{-}\mathfrak{groups}}
\newcommand{\fields}{F\mbox{-}\mathfrak{fields}}
\newcommand{\groups}{\mathfrak{Groups}}
\newcommand{\abelian}{\mathfrak{Ab}}
\newcommand{\p}{\mathfrak{p}}

\newcommand{\Kbar}{\overline{K}}
\newcommand{\Kgenbar}{\overline{K_{\gen}}}

\def\kbar{{\bar{k}}}
\def\Out{{\mathrm{Out}}}
\def\ad{{\mathrm{ad}}}
\def\ttt{{\mathfrak{t}}}
\def\ggg{{\mathfrak{g}}}
\def\G{{\mathbf{G}}}
\def\inn{{\mathrm{inn}}}

\def\AA{{\mathbf{A}}}
\def\BB{{\mathbf{B}}}
\def\CC{{\mathbf{C}}}
\def\DD{{\mathbf{D}}}
\def\GG{{\mathbf{G}}}
\def\ve{{\varepsilon}}
\def\vk{{\varkappa}}
\def\half{{\frac{1}{2}}}
\def\fl{{\mathrm{fl}}}

\def\A{{\mathbb{A}}}

\def\into{\hookrightarrow}

\def\gg{{\mathfrak{g}}}

\def\C{{\mathbb{C}}}
\def\Q{{\mathbb{Q}}}
\def\Z{{\mathbb{Z}}}

\def\coker{{\rm coker}}
\def\Hom{{\rm Hom}}

\def\ba{{\mathbf{a}}}

\def\bo{{\mathbf{1}_q}}
\def\go{{\{0,\bo\}}}

\def\aa{{\mathbf{a}}}
\def\ga{{\langle\aa\rangle}}

\def\bb{{\boldsymbol{\beta}}}
\def\bgen{{\langle\bb\rangle}}

\def\ov{\overline}

\def\Gbar{{\overline{G}}}
\def\Bbar{{\overline{B}}}
\def\Tbar{{\overline{T}}}
\def\im{{\rm im}}

\newcommand{\isoto}{\overset{\sim}{\to}}
\def\SAut{{\mathrm{SAut}}}
\def\lsig{{{}^\sigma}}
\def\sP{{\mathcal{P}}}
\newcommand{\labelto}[1]{\xrightarrow{\makebox[1.5em]{\scriptsize ${#1}$}}}

\def\Gammac{{\Gamma_{\text{\rm{c}}}}}
\def\Stab{{\mathrm{Stab}}}
\def\spl{{\rm spl}}
\def\SAut{{\rm SAut}}

\def\twisted{{\rm twisted}}
\def\sign{{\rm sign\,}}
\def\sS{{{\rm SU}_3}}

\def\SU{{\bf SU}}
\def\SL{{\bf SL}}
\def\GL{{\bf GL}}
\newcommand{\lt}{\mathfrak{t}}
\newcommand{\Lt}{\mathfrak{t}_L}
\def\bbP{{\mathbb{P}}}
\newcommand{\birat}{\overset{\simeq}{\dashrightarrow}}

\def\one{{\boldsymbol{1}}}
\def\T{{\mathfrak{T}}}
\newcommand{\onto}{\twoheadrightarrow}

\def\alphabar{{\overline{\alpha}}}
\def\wbar{{\overline{w}}}
\def\Qbar{{\overline{Q}}}

\def\SO{{\bf SO}}
\def\ms{\medskip\par}
\newcommand{\X}{{\mathcal{X}}}
\newcommand{\XX}{{\textsf{\upshape X}}}
\def\Transp{{\rm Transp}}

\def\xx{{\bf x}}
\def\rk{{\rm rank\,}}
\def\SS{{\mathfrak{S}}}
\def\nubar{\overline{\nu}}

\def\cupdot{\sqcup}
\def\bigcupdot{\bigsqcup}

\newcommand{\sms}{\medskip}

\def\Ubar{{\overline{U}}}
\def\Vbar{{\overline{V}}}
\def\veb{{\overline{\ve}}}

\def\nub{{\boldsymbol{\nu}}}

\def\rhotil{{\tilde\rho}}
\def\sigmatil{{\tilde\sigma}}
\def\tautil{{\tilde\tau}}

\date{\today}

\bigskip
\hfill {\em To Alexander Merkurjev on the occasion of his 60th birthday}

\setcounter{section}{-1}

\section{Introduction}

Let $k$ be a field of characteristic 0 and $\kbar$  a fixed
algebraic closure of $k$. Let  $G$ be a connected linear algebraic
$k$-group. A birational isomorphism $\phi \colon G
\stackrel{\simeq}{\dashrightarrow} \Lie(G)$ is called a {\em Cayley
map} if it is equivariant with respect to the conjugation action of
$G$ on itself and the adjoint action of $G$ on its Lie algebra
$\Lie(G)$, respectively. A linear algebraic $k$-group $G$ is called
{\em Cayley} if it admits a Cayley map, and {\em stably Cayley} if
$G\times_k (\bbG_{{\text{\rm{m}}},k})^r$ is Cayley for some $r \ge
0$. Here $\bbG_{{\text{\rm{m}}},k}$ denotes the multiplicative group
over $k$. These notions were introduced by Lemire, Popov and
Reichstein \cite{LPR}; for a more detailed discussion and numerous
classical examples  we refer the reader to~\cite[Introduction]{LPR}.
The main results of~\cite{LPR} are the classifications of  Cayley
and stably Cayley simple groups over an  algebraically closed field
$k$ of characteristic $0$. Over an arbitrary field $k$ of
characteristic 0 stably Cayley simple $k$-groups, stably Cayley
simply connected semisimple $k$-groups and stably Cayley adjoint
semisimple $k$-groups  were classified in  the paper \cite{BKLR} of
Borovoi, Kunyavski\u\i, Lemire and Reichstein. In the present paper,
building on results of \cite{LPR} and \cite{BKLR}, we classify  all
stably Cayley {\em semisimple}\,\ $k$-groups (not necessarily
simple, or simply connected, or adjoint) over an arbitrary field $k$
of characteristic 0. \sms

By a semisimple (or reductive) $k$-group we always mean a {\em connected} semisimple (or reductive) $k$-group.
We shall need the following result of \cite{BKLR} extending \cite[Theorem 1.28]{LPR}.

\begin{theorem}[{\cite[Theorem 1.4]{BKLR}}] \label{prop:BKLR-abs-simple}
Let $k$ be a field of characteristic $0$ and $G$ an absolutely simple $k$-group.
Then the following conditions are equivalent:
\begin{enumerate}[\upshape(a)]
\item $G$ is stably Cayley over $k$;
\item $G$ is an arbitrary $k$-form of one of the following groups:
\[ \text{$\gSL_3$, $\gPGL_2$, $\gPGL_{2n+1}$ $(n\ge 1)$, $\gSO_n$ $(n \ge 5)$,
$\Sympl_{2n}$ $(n \ge 1)$, $\gG_2$}, \]
or an {\em inner} $k$-form of
$\gPGL_{2n}$ $(n \ge 2)$.
\end{enumerate}
\end{theorem}

In this paper we classify stably Cayley semisimple groups
over an {\em algebraically closed} field $k$ of characteristic 0 (Theorem \ref{thm:product})
and, more generally, over an {\em arbitrary} field $k$ of characteristic 0 (Theorem \ref{thm:product-non-closed}).
Note that Theorem  \ref{thm:product} was conjectured in \cite[Remark 9.3]{BKLR}.

\begin{theorem}\label{thm:product}
Let $k$ be an algebraically closed field of characteristic $0$ and $G$ a semisimple $k$-group.
Then $G$ is stably Cayley if and only if $G$ decomposes into
a direct product $G_1\times_k\dots\times_k G_s$ of its normal subgroups,
where each $G_i$  $(i=1,\dots,s)$ either is a stably Cayley simple $k$-group
(i.e., isomorphic to one of the groups listed in Theorem \ref{prop:BKLR-abs-simple})
or is isomorphic to the stably Cayley semisimple $k$-group $\gSO_4$.
\end{theorem}

\begin{theorem} \label{thm:product-non-closed}
Let $G$ be a  semisimple $k$-group over a field $k$ of characteristic 0 (not necessarily algebraically closed).
Then $G$ is stably Cayley over $k$
if and only if  $G$ decomposes into a direct product $G_1\times_k\dots\times_k G_s$ of its normal $k$-subgroups,
where each $G_i$ $(i=1,\dots,s)$ is isomorphic to the Weil restriction $R_{l_i/k} G_{i,l_i}$
for some finite field extension $l_i/k$,
and each $G_{i,l_i}$ is either a stably Cayley absolutely simple group over $l_i$
(i.e., one of the groups listed in Theorem \ref{prop:BKLR-abs-simple})
or  an  $l_i$-form of the semisimple group $\gSO_4$
(which is always stably Cayley, but is not absolutely simple and can be not $l_i$-simple).
\end{theorem}

Note that  the ``if'' assertions in Theorems \ref{thm:product} and \ref{thm:product-non-closed}
follow immediately from the definitions.
\sms

The rest of the paper is structured as follows.
In Section \ref{sec:1} we recall the definition of a quasi-permutation lattice
and state some known results, in particular, an assertion from \cite[Theorem 1.27]{LPR} that reduces Theorem \ref{thm:product}
to an assertion on lattices.
In Sections \ref{sec:2} and \ref{sec:3} we construct certain families of non-quasi-permutation lattices.
In particular, we correct an inaccuracy in \cite{BKLR}; see Remark \ref{rem:mistake}.
In Section \ref{sec:4} we prove (in the language of lattices)  Theorem \ref{thm:product}
in the special case when $G$  is isogenous to a direct product of simple groups of type $\AA_{n-1}$  with $n\ge 3$.
In Section \ref{sec:5} we prove (again in the language of lattices) Theorem \ref{thm:product} in the general case.
In Section \ref{sec:6}  we deduce Theorem \ref{thm:product-non-closed} from Theorem \ref{thm:product}.
In Appendix \ref{app} we prove in terms of lattices only, that certain quasi-permutation lattices are indeed quasi-permutation.

\section{Preliminaries on quasi-permutation groups and on character lattices}
\label{sec:1}

In this section we gather definitions and known results concerning quasi-permutation lattices,
quasi-invertible lattices and character lattices
that we need for the proofs of Theorems \ref{thm:product} and \ref{thm:product-non-closed}.
For details see \cite[Sections 2 and 10]{BKLR}  and \cite[Introduction]{LPR}.
\sms

\begin{subsec}
By a {\em lattice} we mean a pair $(\Gamma, L)$ where $\Gamma$ is a finite group acting on a finitely generated free abelian group $L$.
We say also that $L$ is a $\Gamma$-lattice.
A $\Gamma$-lattice $L$ is called  a {\em permutation} lattice if it
has a $\bbZ$-basis permuted by $\Gamma$.
Following Colliot-Th\'el\`ene and Sansuc \cite{CTS}, we say that two $\Gamma$-lattices $L$ and $L'$ are
{\em equivalent}, and write $L\sim L'$,
if there exist short exact sequences
\[
0\to L\to E\to P\to 0 \quad \quad {\rm and} \quad \quad
0\to L'\to E\to P'\to 0
\]
with the same $\Gamma$-lattice $E$, where $P$ and $P'$
are permutation $\Gamma$-lattices.
For a proof that this is indeed an equivalence relation
see \cite[Lemma 8, p.~182]{CTS} or \cite[Section 8]{Swan}.
Note that if there exists a short exact sequence of $\Gamma$-lattices
\[
0\to L\to L'\to Q\to 0
\]
where $Q$ is a permutation $\Gamma$-lattice, then,
taking in account the trivial short exact sequence
$$
0\to L'\to L'\to 0\to 0,
$$
we obtain that $L\sim L'$.
If $\Gamma$-lattices $L,L',M,M'$ satisfy $L\sim L'$ and $M\sim M'$,
then clearly $L\oplus M\sim L'\oplus M'$.
\end{subsec}

\begin{definition} \label{def.qp}
A $\Gamma$-lattice $L$ is called a {\em quasi-permutation} lattice
if there exists a short exact sequence
\begin{equation}\label{eq:quasi-perm}
0\to L\to P\to P'\to 0,
\end{equation}
where both $P$ and $P'$ are permutation $\Gamma$-lattices.
\end{definition}

\begin{lemma}[well-known]\label{lem:quasi-p}
A $\Gamma$-lattice $L$ is quasi-permutation if and only if $L\sim 0$.
\end{lemma}

\begin{proof}
If $L$ is quasi-permutation, then sequence \eqref{eq:quasi-perm} together with the trivial short exact sequence
$$
0\to 0\to P\to P\to 0
$$
shows that $L\sim 0$.
Conversely, if $L\sim 0$, then there are short exact sequences
\[
0\to L\to E\to P\to 0 \quad \quad {\rm and} \quad \quad
0\to 0\to E\to P'\to 0,
\]
where $P$ and $P'$ are permutation lattices.
From the second exact sequence we have $E\cong P'$, hence $E$ is a permutation lattice,
and then the first exact sequence shows that $L$ is a quasi-permutation lattice.
\end{proof}

\begin{definition}
A $\Gamma$-lattice $L$ is called {\em quasi-invertible} if it is
a direct summand of a quasi-permutation $\Gamma$-lattice.
\end{definition}

Note that if a $\Gamma$-lattice $L$ is not quasi-invertible, then it is not quasi-permutation.

\begin{lemma}[well-known]
 If a $\Gamma$-lattice $L$ is quasi-permutation (resp., quasi-invertible) and $L'\sim L$, then
$L'$ is quasi-permutation (resp., quasi-invertible) as well.
\end{lemma}
\begin{proof}
If $L$ is quasi-permutation, then using Lemma \ref{lem:quasi-p} we see that $L'\sim L\sim 0$,
hence $L'$ is quasi-permutation.
If $L$ is quasi-invertible, then $L\oplus M$ is quasi-permutation for some $\Gamma$-lattice $M$, and
by Lemma \ref{lem:quasi-p} we have $L\oplus M\sim 0$.
We see that $L'\oplus M\sim L\oplus M\sim 0$, and
by Lemma \ref{lem:quasi-p} we obtain that
$L'\oplus M$ is quasi-permutation,
hence $L'$ is quasi-invertible.
\end{proof}

\sms

Let $\Z[\Gamma]$ denote the group ring of a finite group $\Gamma$.
We define the $\Gamma$-lattice $J_\Gamma$
by the exact sequence
$$
0\to\Z\labelto{N}\Z[\Gamma]\to J_\Gamma\to 0,
$$
where $N$ is the norm map, see \cite[before Lemma 10.4]{BKLR}.
We refer to \cite[Proposition 10.6]{BKLR} for
a proof of the following result, due to Voskresenski\u\i\
\cite[Corollary of Theorem 7]{Voskresenskii70}:

\begin{proposition}\label{prop:J-Gamma}
Let $\Gamma=\Z/p\Z\times \Z/p\Z$, where $p$ is a prime.
Then the $\Gamma$-lattice $J_\Gamma$ is not quasi-invertible.
\end{proposition}

Note that if $\Gamma=\Z/2\Z\times \Z/2\Z$, then $\rk J_\Gamma=3$.
\sms

We shall use the following lemma from \cite{BKLR}:

\begin{lemma}[{\cite[Lemma 2.8]{BKLR}}]
\label{lem:2.8}
  Let $W_1,\dots,W_m$ be finite groups.
  For each $i=1,\dots,m$, let $V_i$ be a finite-dimensional
  $\Q$-representation of $W_i$. Set $V:=V_1\oplus\dots\oplus V_m$.
  Suppose $L\subset V$ is a free abelian subgroup, invariant under
  $W :=W_1\times\dots\times W_m$.
 If $L$ is a quasi-permutation $W$-lattice,
  then for each $i=1, \dots, m$
the intersection $L_i:=L\cap V_i$ is a quasi-permutation $W_i$-lattice.
  \end{lemma}

We shall need the notion, due to \cite{LPR} and \cite{BKLR}, of the character lattice of a reductive $k$-group $G$
over a field $k$. Let $\kbar$ be a separable closure of $k$.
Let $T\subset G$ be a maximal torus (defined over $k$). Set $\Tbar=T\times_k \kbar$, $\Gbar=G\times_k \kbar$.
Let $\XX(\Tbar)$ denote the character group of $\Tbar:=T\times_k \kbar$.
Let $W=W(\Gbar,\Tbar):=\mathcal{N}_G(\Tbar)/\Tbar$ denote the Weyl group, it acts  on $\XX(\Tbar)$.
Consider the canonical Galois action on $\XX(\Tbar)$, it defines a homomorphism $\Gal(\kbar/k)\to\Aut\XX(\Tbar)$.
The image $\im\,\rho\subset \Aut\XX(\Tbar)$ normalizes $W$, hence $\im\, \rho\,\cdot W$ is a subgroup of $\Aut\XX(\Tbar)$.
By the character lattice of $G$ we mean the pair $\X(G):=(\im\, \rho\cdot W,\,\XX(\Tbar))$
(up to an isomorphism it does not depend on the choice of $T$).
In particular, if $k$ is algebraically closed, then $\X(G)=(W,\XX(T))$.
\sms

We shall reduce Theorem \ref{thm:product} to an assertion about quasi-permutation lattices using the following result due to \cite{LPR}:

\begin{proposition}[{\cite[Theorem 1.27]{LPR}, see also \cite[Theorem 1.3]{BKLR}}]
\label{prop:LPR,BKLR}
A reductive group $G$ over an algebraically closed field $k$ of characteristic $0$ is stably Cayley
if and only if its character lattice $\X(G)$ is quasi-permutation, i.e., $\XX(T)$ is a quasi-permutation $W(G,T)$-lattice.
\end{proposition}

We shall use the following result due to Cortella and Kunyavski\u\i\ \cite{CK} and to Lemire, Popov and Reichstein \cite{LPR}.

\begin{proposition}[\cite{CK}, \cite{LPR}]\label{prop:CK,LPR}
Let $D$ be a connected Dynkin diagram.
Let $R=R(D)$ denote the corresponding root system,
$W=W(D)$ denote the Weyl group,
$Q=Q(D)$ denote the root lattice, and $P=P(D)$ denote the weight lattice.
We say that $L$ is an {\em intermediate lattice  between $Q$ and $P$}
if  $Q\subset L\subset P$ (note that $L=Q$ and $L=P$ are possible).
Then the following list gives (up to an isomorphism) all the pairs $(D,L)$ such that
$L$ is a quasi-permutation intermediate $W(D)$-lattice between $Q(D)$ and $P(D)$:
$$Q(\AA_n),\ Q(\BB_n),\  P(\CC_n),\ \X(\SO_{2n})\ \text{(then $D=\DD_n$)},$$
or $D$ is any connected Dynkin diagram of rank $1$ or $2$ (i.e. $\AA_1$, $\AA_2$, $\BB_2$, or $\GG_2$)  and
$L$ is any lattice between $Q(D)$ and $P(D)$, (i.e., either $L=P(D)$ or $L=Q(D)$).
\end{proposition}

\begin{proof}
The positive result (the assertion that the lattices in the list are
indeed quasi-permutation) follows from the assertion that the
corresponding groups are stably Cayley (or that their generic tori
are stably rational), see the references in \cite{CK}, Section 3.
See Appendix \ref{app} below for a proof of this positive result in
terms of lattices only. The difficult part  of Proposition
\ref{prop:CK,LPR} is the negative result (the assertion that all the
other lattices are not quasi-permutation). This was proved in
\cite[Theorem 0.1]{CK} in the cases when $L=Q$ or $L=P$, and in
\cite[Propositions 5.1 and 5.2]{LPR} in the cases when $Q\subsetneq
L\subsetneq P$ (this can happen only when $D=\AA_n$ or $D=\DD_n$).
\end{proof}

\begin{remark}
It follows from Proposition \ref{prop:CK,LPR} that, in particular,
the following lattices are quasi-permutation: $Q(\AA_1),\ P(\AA_1),\
P(\AA_2),\ P(\BB_2),\ Q(\CC_2)$, $Q(\GG_2)=P(\GG_2)$,
$Q(\DD_3)=Q(\AA_3),\ \X(\SL_4/\mu_2)=\X(\SO_6)$.
\end{remark}

\section{A family of non-quasi-permutation lattices}
\label{sec:2}

In this section we construct a family of non-quasi-permutation (even non-quasi-invertible) lattices.

\begin{subsec}\label{subsec:one-vector}
We consider a Dynkin diagram $D\cupdot\Delta$ (disjoint union).
We assume that $D=\bigcupdot_{i\in I} D_i$ (a finite disjoint union),
where each $D_i$ is of type $\BB_{l_i}\ (l_i\ge 1)$ or $\DD_{l_i}\ (l_i\ge 2)$ (and where $\BB_1=\AA_1$,
$\BB_2=\CC_2$, $\DD_2=\AA_1\cupdot \AA_1$, and $\DD_3=\AA_3$ are permitted).
We denote by $m$ the cardinality of the finite index set $I$.
We assume that $\Delta=\bigcupdot_{\iota=1}^\mu \Delta_\iota$ (disjoint union),
where  $\Delta_\iota$ is of type $\AA_{2n_\iota-1}$, $n_\iota\ge 2$ ($\AA_3=\DD_3$ is permitted).
We assume that $m\ge 1$ and $\mu\ge 0$ (in the case $\mu=0$ the diagram $\Delta$ is empty).
\sms

For each $i\in I$   we realize the root system $R(D_i)$ of type $\BB_{l_i}$ or $\DD_{l_i}$
in the standard way in the space $V_i:=\R^{l_i}$ with basis $(e_s)_{s\in S_i}$
where $S_i$ is an index set consisting of $l_i$ elements;
cf. \cite[Planche II]{Bourbaki} for $\BB_l$ $(l\ge 2)$ (the relevant formulas for $\BB_1$  are similar)
and \cite[Planche IV]{Bourbaki} for $\DD_l$ $(l\ge 3)$ (again, the relevant formulas for $\DD_2$ are similar).
Let $M_i\subset V_i$ denote the lattice generated by the basis vectors $(e_s)_{s\in S_i}$.
Let $P_i\supset M_i$ denote the weight lattice of the root system $D_i$.
Set $S=\bigcupdot_i S_i$ (disjoint union).
Consider the vector space $V=\bigoplus_i V_i$ with basis $(e_s)_{s\in S}$.
Let $M_D\subset V$ denote the lattice generated by the basis vectors $(e_s)_{s\in S}$, then $M_D=\bigoplus_i M_i$.
Set $P_D=\bigoplus_i P_i$.
\sms

For each $\iota=1,\dots,\mu$ we realize the root system $R(\Delta_\iota)$ of type $\AA_{2n_\iota-1}$
in the standard way in  the subspace $V_\iota$ of vectors with zero sum of the coordinates
in  the space $\R^{2n_\iota}$ with basis
$\ve_{\iota,1},\dots, \ve_{\iota,2n_\iota}$; cf. \cite[Planche I]{Bourbaki}.
Let $Q_\iota$ be the root lattice of $R(\Delta_\iota)$ with basis
$\ve_{\iota,1}-\ve_{\iota,2},\ \ve_{\iota,2}-\ve_{\iota,3},\ \dots, \ \ve_{\iota,2n_\iota-1}-\ve_{\iota,2n_\iota}$,
and let $P_\iota\supset Q_\iota$ be the weight lattice of $R(\Delta_\iota)$.
Set $Q_\Delta=\bigoplus_\iota Q_\iota$,
 $P_\Delta=\bigoplus_\iota P_\iota$.
\sms

Set
$$
W:= \prod_{i\in I} W(D_i)\times\prod_{\iota=1}^\mu W(\Delta_\iota),
\quad L'=M_D\oplus Q_\Delta=\bigoplus_{i\in I} M_i \oplus
\bigoplus_{\iota=1}^\mu Q_\iota,
$$
then $W$ acts on $L'$ and on $L'\otimes_\Z \R $.
For each $i$ consider the vector
$$
x_i=\sum_{s\in S_i} e_s\in  M_i,
$$
then $\half x_i\in P_i$.
For each $\iota$ consider the vector
$$
\xi_\iota=\ve_{\iota,1}-\ve_{\iota,2}+\ve_{\iota,3}-\ve_{\iota,4}+\dots+\ve_{\iota,2n_\iota-1}-\ve_{\iota,2n_\iota}\in  Q_\iota,
$$
then $\half\xi_\iota\in P_\iota$; see \cite[Planche I]{Bourbaki}.
Write
$$
\xi'_\iota=\ve_{\iota,1}-\ve_{\iota,2},\quad  \xi''_\iota=\ve_{\iota,3}-\ve_{\iota,4}+\dots+\ve_{\iota,2n_\iota-1}-\ve_{\iota,2n_\iota},
$$
then $\xi_\iota=\xi'_\iota+\xi''_\iota$.
Consider the vector
$$
v=\ \half\sum_{i\in I} x_i+\half\sum_{\iota=1}^\mu \xi_{\iota}\ =\ \half\sum_{s\in S}e_s+\half\sum_{\iota=1}^\mu \xi_{\iota}\in P_D\oplus P_\Delta.
$$
Set
\begin{equation}\label{eq:L'v}
 L=\langle L', v\rangle,
\end{equation}
then $[L:L']=2$ because $v\in\half L'\smallsetminus L'$.
Note that the sublattice $L\subset  P_D\oplus P_\Delta$ is $W$-invariant.
Indeed, the group $W$ acts on $(P_D\oplus P_\Delta)/(M_D\oplus Q_\Delta)$ trivially.
\end{subsec}

\begin{proposition}\label{prop:one-vector}
We assume that $m\ge 1$, $m+\mu \ge 2$.
If $\mu=0$, we assume that  not all of $D_i$ are of types $\BB_1$ or $\DD_2$.
Then the $W$-lattice $L$ as in \eqref{eq:L'v} is not quasi-invertible, hence not quasi-permutation.
\end{proposition}

\begin{proof}
We consider a group $\Gamma=\{e,\gamma_1,\gamma_2,\gamma_3\}$ of order 4, where $\gamma_1,\gamma_2,\gamma_3$ are of order 2.
The idea of our proof is to construct an embedding
\begin{equation*} \label{e.gamma}
j\colon\Gamma\to W
\end{equation*}
in such a way that $L$, viewed as a $\Gamma$-lattice, is equivalent
to its $\Gamma$-sublattice $L_1$, and $L_1$ is isomorphic to a
direct sum of a $\Gamma$-sublattice $L_0 \simeq J_{\Gamma}$ of rank
3 and a number of $\Gamma$-lattices of rank $1$. Since by
Proposition \ref{prop:J-Gamma} $J_\Gamma$ is not quasi-invertible,
this will imply that $L_1$ and $L$ are not quasi-invertible
$\Gamma$-lattices, and hence $L$ is not quasi-invertible as a
$W$-lattice. We shall now fill in the details of this argument in
four steps.
\sms

\noindent
  {\em Step 1. }
We begin by partitioning each $S_i$ for $i\in I$ into three (non-over\-lapping)
subsets $S_{i, 1}$, $S_{i, 2}$ and $S_{i, 3}$,
subject to the requirement that
\begin{equation} \label{align:vk}
|S_{i, 1}| \equiv
|S_{i, 2}| \equiv
|S_{i, 3}|
\equiv l_i \text{ (mod ${2}$) if  $D_i$ is of type
$\DD_{l_i}$.}
\end{equation}
We then set $U_1$ to be the union of the $S_{i, 1}$,
$U_2$ to be the union of the $S_{i, 2}$,
and $U_3$ to be the union of the $S_{i, 3}$,
as $i$ runs over $I$.

\begin{lemma} \label{lem.partitions}
\begin{itemize}
\item[\rm(i)] If $\mu\ge 1$, the subsets $S_{i, 1}$, $S_{i, 2}$ and $S_{i,
3}$ of $S_i$ can be chosen, subject to~\eqref{align:vk}, so that
$U_1\neq\emptyset$.
\item[\rm(ii)] If $\mu=0$ (and $m\ge 2$), the
subsets $S_{i, 1}$, $S_{i, 2}$ and $S_{i, 3}$ of $S_i$ can be
chosen, subject to~\eqref{align:vk}, so that $U_1, U_2, U_3 \neq
\emptyset$.
\end{itemize}
\end{lemma}

To prove the lemma, first consider case (i). For all $i$ such that
$D_i$ is of type $\DD_{l_i}$ with {\em odd} $l_i$,  we  partition
$S_i$ into three non-empty subsets of odd cardinalities. For all the
other $i$ we take $S_{i,1}=S_i$, $S_{i,2}=S_{i,3}=\emptyset$. Then
$U_1\neq\emptyset$ (note that $m\ge 1$) and  \eqref{align:vk} is
satisfied. \sms

In case (ii), if one of the $D_i$ is of type $\DD_{l_i}$ where $l_i
\ge 3$ is {\em odd,} then we  partition $S_i$ for each such  $D_i$
into three non-empty subsets of odd cardinalities. We partition all
the other $S_{i}$  as follows:
\begin{equation} \label{e.Si}
  \text{$S_{i, 1} = S_{i, 2} = \emptyset$ and $S_{i, 3} = S_i$.}
\end{equation}
Clearly $U_1, U_2, U_3 \neq \emptyset$ and  \eqref{align:vk} is satisfied.
\sms

If there is no $D_i$ of type $\DD_{l_i}$ with  odd $l_i \ge 3$, but
one of the $D_i$, say for $i=i_0$, is $\DD_l$ with {\em even} $l \ge
4$,   then we partition $S_{i_0}$ into two non-empty subsets
$S_{i_0,1}$ and $S_{i_0,2}$ of even cardinalities, and set
$S_{i_0,3}=\emptyset$. We partition the sets $S_i$  for $i\neq i_0$
as in~\eqref{e.Si} (note that by our assumption $m\ge 2$). Once
again, $U_1, U_2, U_3 \neq \emptyset$ and  \eqref{align:vk} is
satisfied. \sms

If there is no $D_i$ of type $\DD_{l_i}$ with $l_i \ge 3$ (odd or even),
but one of the $D_i$, say for $i=i_0$, is of type $\BB_l$ with $l\ge 2$,
we partition $S_{i_0}$ into two non-empty subsets  $S_{i_0,1}$ and $S_{i_0,2}$, and set $S_{i_0,3}=\emptyset$.
We partition the sets $S_i$ for $i\neq i_0$ as in~\eqref{e.Si}  (again, note that $m\ge 2$).
Once again, $U_1, U_2, U_3 \neq \emptyset$ and  \eqref{align:vk} is satisfied.
\sms

Since by our assumption not all of $D_i$ are of type $\BB_1$ or $\DD_2$,
we have exhausted all the cases.
This completes the proof of Lemma~\ref{lem.partitions}.
\qed
\sms

\noindent
{\em Step 2.}
We continue proving Proposition \ref{prop:one-vector}.
We construct an embedding $\Gamma\into W$.
\sms

For $s\in S$ we denote by $c_s$  the automorphism  of $L$ taking the basis vector $e_s$ to $-e_s$ and fixing all the other basis vectors.
For $\iota=1,\dots,\mu$  we define
$\tau_\iota^{(12)}=\Transp((\iota,1),(\iota,2))\in W_\iota$ (the transposition of the basis vectors $\ve_{\iota,1}$ and $\ve_{\iota,2}$).
Set
$$\tau_\iota^{>2}=\Transp((\iota,3),(\iota,4))\cdot\ \cdots\ \cdot\Transp((\iota,2n_\iota-1),(\iota,2n_\iota))\in W_\iota.
$$

Write $\Gamma=\{e,\gamma_1,\gamma_2,\gamma_3\}$ and define an embedding $j\colon \Gamma\into W$ as follows:
\begin{align*}
j(\gamma_1)&=\prod_{s\in S\smallsetminus U_1} c_s\ \cdot\ \prod_{\iota=1}^\mu \tau_\iota^{(12)}\tau_\iota^{>2}\,;\\
j(\gamma_2)&=\prod_{s\in S\smallsetminus U_2} c_s\ \cdot\ \prod_{\iota=1}^\mu \tau_\iota^{(12)}\,;\\
j(\gamma_3)&=\prod_{s\in S\smallsetminus U_3} c_s\ \cdot\ \prod_{\iota=1}^\mu \tau_\iota^{>2}\,.
\end{align*}

Note that if $D_i$ is of type $\DD_{l_i}$, then  by \eqref{align:vk}  for $\vk=1,2,3$ the cardinality
$\#(S_i\smallsetminus S_{i,\vk})$ is even,
hence the product of $c_s$ over $s\in S_i\smallsetminus S_{i,\vk}$ is contained in $ W(D_i)$ for all such $i$, and therefore,
 $j(\gamma_\vk)\in W$.
Since
$j(\gamma_1)$, $j(\gamma_2)$ and $j(\gamma_3)$ commute, are of order 2,
and $j(\gamma_1) j(\gamma_2)=j(\gamma_3)$,
we see that $j$ is a homomorphism.
If $\mu\ge 1$, then, since  $2n_1\ge 4$,  clearly $j(\gamma_\vk)\neq 1$ for $\vk=1,2,3$, hence $j$ is an embedding.
If $\mu=0$, then the sets $S\smallsetminus U_1$, $S\smallsetminus U_2$ and $S\smallsetminus U_3$ are nonempty,
and again  $j(\gamma_\vk)\neq 1$ for $\vk=1,2,3$, hence $j$ is an embedding.
\sms

\noindent
{\em Step 3.}  We construct a $\Gamma$-sublattice $L_0$ of rank 3.
Write a vector  $\xx\in L$ as
$$
\xx=\sum_{s\in S} b_s e_s + \sum_{\iota=1}^\mu\sum_{\nu=1}^{2n_\iota}\beta_{\iota,\nu}\ve_{\iota,\nu},
$$
where $b_s\,,\beta_{\iota,\nu}\in \half\Z$.
 Set $n'=\sum_{\iota=1}^\mu (n_\iota-1)$.
Define a $\Gamma$-equivariant homomorphism
$$
\phi\colon L\to \Z^{n'},\quad \xx\mapsto (\beta_{\iota,2\lambda-1}+\beta_{\iota,2\lambda})_{\iota=1,\dots,\mu,\ \lambda=2,\dots,n_\iota}
$$
(we skip $\lambda=1$).
We obtain a short exact sequence of $\Gamma$-lattices
$$
0\to L_1\to L\labelto{\phi} \Z^{n'}\to 0,
$$
where $L_1:=\ker\phi$.
Since $\Gamma$ acts trivially on $\Z^{n'}$,
we have $L_1\sim L$.
Therefore, it suffices to show that $L_1$ is not quasi-invertible.
\sms

Recall that
$$
v=\ \half\sum_{s\in S}e_s+\half\sum_{\iota=1}^\mu \xi_{\iota}.
$$
Set $v_1=\gamma_1\cdot v,\ v_2=\gamma_2\cdot v,\ v_3=\gamma_3\cdot v.$
Set
$$
L_0=\langle v, v_1,v_2,v_3\rangle.
$$
 We have
$$
v_1=\ \half\sum_{s\in U_1}e_s  -\half\sum_{s\in U_2\cup U_3}e_s    -\half\sum_{\iota=1}^\mu \xi_{\iota},
$$
whence
\begin{equation}\label{e:sum1}
v+v_1=\sum_{s\in U_1}e_s.
\end{equation}
We have
$$
v_2=\ \half\sum_{s\in U_2}e_s  -\half\sum_{s\in U_1\cup U_3}e_s   +\half\sum_{\iota=1}^\mu (-\xi'_{\iota} + \xi''_{\iota}),
$$
whence
\begin{equation}\label{e:sum2}
v+v_2=\sum_{s\in U_2}e_s + \sum_{\iota=1}^\mu \xi''_{\iota} .
\end{equation}
We have
$$
v_3=\ \half\sum_{s\in U_3}e_s  -\half\sum_{s\in U_1\cup U_2}e_s   + \half\sum_{\iota=1}^\mu (\xi'_{\iota} - \xi''_{\iota}),
$$
whence
\begin{equation}\label{e:sum3}
v+v_3=\sum_{s\in U_3}e_s + \sum_{\iota=1}^\mu \xi'_{\iota}.
\end{equation}

Clearly, we have
$$
v+v_1+v_2+v_3=0.
$$
Since the set $\{v,v_1,v_2,v_{3}\}$ is the orbit of $v$ under $\Gamma$,
the sublattice $L_0=\langle v, v_1,v_2,v_3\rangle\subset L$ is $\Gamma$-invariant.
If $\mu\ge 1$, then $U_1\neq\emptyset$, and we see from \eqref{e:sum1}, \eqref{e:sum2} and \eqref{e:sum3} that $\rk L_0 \ge 3$.
If $\mu=0$, then $U_1,U_2,U_3\neq\emptyset$, and again we see from \eqref{e:sum1}, \eqref{e:sum2} and \eqref{e:sum3} that $\rk L_0 \ge 3$.
Thus $\rk L_0=3$ and
$L_0\simeq J_\Gamma$, whence   by Proposition \ref{prop:J-Gamma} $L_0$ is not quasi-invertible.
\sms

\noindent
{\em Step 4.}
We show that  $L_0$ is a direct summand of $L_1$. Set $m'=|S|$.
\sms

First assume that $\mu\ge 1$.
Choose $u_1\in U_1\subset S$. Set $S'=S\smallsetminus \{u_1\}$.
For each $s\in S'$ (i.e., $s\neq u_1$) consider the one-dimensional (i.e., of rank 1) lattice $X_s=\langle e_s\rangle$.
We obtain $m'-1$ $\Gamma$-invariant one-dimensional sublattices of $L_1$.
\sms

Denote by $\Upsilon$ the set of pairs $(\iota,\lambda)$ such that $1\le\iota\le\mu$,
  $1\le \lambda\le n_\iota$, and if $\iota=1$, then $\lambda\neq 1,2$.
For each $(\iota,\lambda)\in\Upsilon$ consider the one-dimensional lattice
$$
\Xi_{\iota,\lambda}=\langle \ve_{\iota,2\lambda-1}-\ve_{\iota,2\lambda}\rangle.
$$
We obtain $-2+\sum_{\iota=1}^\mu n_\iota$  one-dimensional $\Gamma$-invariant  sublattices of $L_1$.
\sms

We show that
\begin{equation}\label{e:d-sum}
L_1=L_0\oplus \bigoplus_{s\in S'} X_s \oplus \bigoplus_{(\iota,\lambda)\in\Upsilon}\Xi_{\iota,\lambda}.
\end{equation}
Set $L'_1=\langle L_0, (X_s)_{s\neq u_1}, (\Xi_{\iota,\lambda})_{(\iota,\lambda)\in\Upsilon}\rangle$, then
\begin{equation}\label{eq:dim-count1}
\rk L'_1\leq 3+(m'-1)-2+\sum_{\iota=1}^\mu n_\iota=
m'+\sum_{\iota=1}^\mu (2n_\iota-1)-\sum_{\iota=1}^\mu (n_\iota-1)=\rk L_1.
\end{equation}
Therefore, it suffices to check that $L'_1\supset L_1$.
The set
$$\{v\} \cup  \{e_s\ |\ s\in S\}\cup\{\ve_{\iota,2\lambda-1}-\ve_{\iota,2\lambda}\ |\ 1\le\iota\le\mu,1\le \lambda\le n_\iota\}$$
is a set of generators of $L_1$.
By construction $v,v_1,v_2,v_3\in L_0\subset L'_1$.
We have  $e_s\in X_s\subset L'_1$ for $s\neq u_1$.
By \eqref{e:sum1} $\sum_{s\in U_1}e_s\in L'_1$, hence $e_{u_1}\in L'_1$.
By construction
$$
\ve_{\iota,2\lambda-1}-\ve_{\iota,2\lambda}\in L'_1,\  \text{ for all } (\iota,\lambda)\neq (1,1),(1,2) \,.
$$
From \eqref{e:sum3} and \eqref{e:sum2} we see that
$$
\sum_{\iota=1}^\mu (\ve_{\iota,1}-\ve_{\iota,2})\in L'_1,\ \ \sum_{\iota=1}^\mu \xi''_\iota\in L'_1\,.
$$
Thus
$$
\ve_{1,1}-\ve_{1,2}\in L'_1,\ \ \ve_{1,3}-\ve_{1,4}\in L'_1\,.
$$
We conclude that $L'_1\supset L_1$, hence $L_1=L'_1$.
From dimension count \eqref{eq:dim-count1} we see that \eqref{e:d-sum} holds.
\sms

Now assume that $\mu=0$. Then for each $\vk=1,2,3$ we choose an element $u_\vk\in U_\vk$ and set $U'_\vk=U_\vk\smallsetminus\{u_\vk\}$.
We set $S'=U'_1\cup U'_2\cup U'_3=S\smallsetminus\{u_1,u_2,u_3\}$.
Again for $s\in S'$ (i.e., $s\neq u_1,u_2,u_3$) consider the one-dimensional lattice $X_s=\langle e_s\rangle$.
We obtain $m'-3$  one-dimensional $\Gamma$-invariant  sublattices of $L_1=L$.
We show that
\begin{equation}\label{e:e-sum}
L_1=L_0\oplus \bigoplus_{s\in S'} X_s\,.
\end{equation}
Set $L'_1=\langle L_0, (X_s)_{s\in S'} \rangle$, then
\begin{equation}\label{eq:dim-count2}
\rk L'_1\leq 3+m'-3=m'=\rk L_1.
\end{equation}
Therefore, it suffices to check that $L'_1\supset L_1$.
The set $\{v\} \cup \{e_s\ |\ s\in S\}$ is a set of generators of $L_1=L$.
By construction
$v,v_1,v_2,v_3\in L_1'$ and $e_s\in L_1'$ for $s\neq u_1,u_2,u_3$.
We see from \eqref{e:sum1}, \eqref{e:sum2}, \eqref{e:sum3} that
$e_s\in L_1'$ also for $s=u_1,u_2,u_3$.
Thus $L'_1\supset L_1$, hence $L'_1= L_1$.
From  dimension count  \eqref{eq:dim-count2} we see that \eqref{e:e-sum} holds.
\sms

We see that in both cases $\mu\ge 1$ and $\mu=0$, the sublattice $L_0$ is a direct summand of $L_1$.
Since by Proposition \ref{prop:J-Gamma} $L_0$ is not quasi-invertible as a $\Gamma$-lattice,
it follows that $L_1$ and $L$ are not quasi-invertible as $\Gamma$-lattices.
Thus $L$ is not quasi-invertible as a $W$-lattice.
This completes the proof of Proposition \ref{prop:one-vector}.
\end{proof}

\begin{remark}
Since $\Sha^2(\Gamma,J_\Gamma)\cong\Z/2\Z$ (Voskresenski\u\i, see \cite[Section 10]{BKLR} for the notation and the result),
our argument shows that $\Sha^2(\Gamma,L)\cong\Z/2\Z$.
\end{remark}

\begin{remark}\label{rem:mistake}
The proof of \cite[Lemma 12.3]{BKLR} (which is a version with
$\mu=0$ of Lemma \ref{lem.partitions} above) contains an inaccuracy,
though the lemma as stated is correct. Namely, in \cite{BKLR} we
write that if there exists $i$ such that $\Delta_i$ is of type
$\DD_{l_i}$ where $l_i \ge 3$ is {odd,} then  we  partition  $S_i$
for {\em one} such $i$ into three non-empty subsets $S_{i,1}$,
$S_{i,2}$ and $S_{i,3}$ of odd cardinalities, and we partition all the
other $S_i$ as in \cite[(12.4)]{BKLR}. However, this partitioning of
the sets $S_i$ into three subsets does not satisfy
\cite[(12.3)]{BKLR} for {\em other}  $i$ such that $\Delta_i$ is of
type $\DD_{l_i}$ with {odd} $l_i$. This inaccuracy can be easily
corrected: we should partition  $S_i$ for {\em each} $i$ such that
$\Delta_i$ is of type $\DD_{l_i}$ with {odd} $l_i$ into three
non-empty subsets of odd cardinalities.
\end{remark}

\section{More non-quasi-permutation lattices}
\label{sec:3}

In this section we construct another family of non-quasi-permutation lattices.

\begin{subsec}\label{subsec:qp-An-nu}
For $i=1,\dots,r$ let $Q_i=\Z \AA_{n_i-1}$ and $P_i=\Lambda_{n_i}$ denote the root lattice and the weight lattice of $\SL_{n_i}$, respectively,
and let $W_i=\mathfrak{S}_{n_i}$ denote the corresponding Weyl group (the symmetric group on $n_i$ letters) acting on $P_i$ and $Q_i$.
Set $F_i=P_i/Q_i$, then $W_i$ acts trivially on $F_i$.
Set
$$
Q=\bigoplus_{i=1}^r Q_i,\quad P=\bigoplus_{i=1}^r P_i,\quad W=\prod_{i=1}^r W_i,
$$
then $Q\subset P$ and the Weyl group $W$ acts on $Q$ and $P$.
Set
$$F=P/Q=\bigoplus_{i=1}^r F_i,$$
then $W$ acts trivially on $F$.
\sms

We regard $Q_i=\Z \AA_{n_i-1}$ and $P_i=\Lambda_{n_i}$ as the lattices described in  Bourbaki \cite[Planche I]{Bourbaki}.
Then we have an isomorphism $F_i\cong \Z/n_i\Z$.
Note that for each $1\le i\le r$, the set $\{\alpha_{\vk,i}\ |\ 1\le\vk\le n_i-1\}$ is a $\Z$-basis of $Q_i$.
\sms

Set $c=\gcd(n_1,\dots,n_r)$; we assume that $c>1$.
Let $d>1$ be a divisor of $c$.
For each $i=1,\dots,r$, let $\nu_i\in \Z$ be such that $1\le \nu_i<d$, $\gcd(\nu_i,d)=1$, and assume that $\nu_1=1$.
We write $\nub=(\nu_i)_{i=1}^r\in\Z^r$.
Let $\overline{\nub}$ denote the image of $\nub$ in  $(\Z/d\Z)^r$.
Let $S_\nub\subset (\Z/d\Z)^r\subset\bigoplus_{i=1}^r \Z/n_i\Z=F$ denote the cyclic subgroup of order $d$ generated by $\overline{\nub}$.
Let $L_\nub$ denote the preimage of $S_\nub\subset F$ in $P$ under the canonical epimorphism $P\onto F$, then $Q\subset L_\nub\subset P$.
\end{subsec}

\begin{proposition}\label{prop:latticeL}
Let $W$ and the $W$-lattice $L_\nub$ be as in Subsection \ref{subsec:qp-An-nu}.
In the case $d=2^s$ we assume that $\sum n_i>4$.
Then  $L_\nub$  is not quasi-permutation.
\end{proposition}

This proposition follows from Lemmas \ref{lem:equivalent} and \ref{lem:not-qp} below.

\begin{lemma}\label{lem:equivalent}
Let $p| d$ be a prime.
Then for any subgroup $\Gamma\subset W$ isomorphic to $(\Z/p\Z)^m$ for some natural $m$,
the $\Gamma$-lattices $L_\nub$ and $L_\one:=L_{(1,\dots,1)}$ are equivalent
for any $\nub=(\nu_1,\dots,\nu_r)$ as above (in particular, we assume that $\nu_1=1$).
\end{lemma}

Note that this lemma is trivial when $d=2$.

\begin{subsec}
We compute the lattice $L_\nub$ explicitly. First let $r=1$. We have
$Q=Q_1$, $P=P_1$. Then $P_1$ is generated by $Q_1$ and an element
$\omega\in P_1$ whose image in $P_1/Q_1$ is of order $n_1$. We may
take
$$
\omega=\frac{1}{n_1}[(n_1-1)\alpha_1+(n_1-2)\alpha_2+\dots +2\alpha_{n_1-2}+\alpha_{n_1-1}],
$$
where $\alpha_1,\dots,\alpha_{n_1-1}$ are the simple roots, see \cite[Planche I]{Bourbaki}.
There exists exactly one intermediate lattice $L$ between $Q_1$ and $P_1$ such that $[L:Q_1]=d$, and it is generated by $Q_1$ and the element
$$
w=\frac{n_1}{d}\omega=\frac{1}{d}[(n_1-1)\alpha_1+(n_1-2)\alpha_2+\dots +2\alpha_{n_1-2}+\alpha_{n_1-1}].
$$
Now for any natural $r$, the lattice $L_\nub$ is generated by $Q$ and the element
$$
w_\nub=\frac{1}{d}\sum_{i=1}^r\nu_i[(n_i-1)\alpha_{1,i}+(n_i-2)\alpha_{2,i}+\dots +2\alpha_{n_i-2,i}+\alpha_{n_i-1,i}].
$$
In particular, $L_\one$ is generated by $Q$ and
$$
w_\one=\frac{1}{d}\sum_{i=1}^r[(n_i-1)\alpha_{1,i}+(n_i-2)\alpha_{2,i}+\dots +2\alpha_{n_i-2,i}+\alpha_{n_i-1,i}].
$$
\end{subsec}

\begin{subsec} {\em Proof of Lemma \ref{lem:equivalent}.}
Recall that $L_\nub=\langle Q, w_\nub\rangle$ with
$$
Q=\langle\alpha_{\vk,i}\rangle, \quad\text{where}\quad  i=1,\dots, r,\ \vk=1,\dots,n_i-1.
$$
Set $Q_\nub=\langle\nu_i\alpha_{\vk,i}\rangle$.
Denote by $\T_\nub$ the endomorphism of $Q$ that acts on $Q_i$ by multiplication by $\nu_i$.
We have $Q_\one=Q,\ Q_\nub=\T_\nub Q_\one,\ w_\nub=\T_\nub w_\one$. Consider
$$\T_\nub L_\one=\langle Q_\nub,w_\nub\rangle.$$
Clearly the $W$-lattices $L_\one$ and $\T_\nub L_\one$ are
isomorphic. We have an embedding of $W$-lattices $Q\into L_\nub$, in
particular, an embedding $Q\into L_\one$, which induces an embedding
$\T_\nub Q\into \T_\nub L_\one$. Set $M_\nub=L_\nub/{\T_\nub
L_\one}$, then we obtain a homomorphism of $W$-modules $Q/\T_\nub
Q\to M_\nub$, which is an isomorphism  by Lemma \ref{lem:quotient}
below. \sms

Now let $p|d$ be a prime. Let $\Gamma\subset W$ be a subgroup
isomorphic to $(\Z/p\Z)^m$ for some natural $m$.
As in \cite[Proof of Proposition 2.10]{LPR},
we use Roiter's version  \cite[Proposition 2]{Roiter} of Schanuel's lemma.
We have exact sequences of $\Gamma$-modules
\begin{gather*}
0\to {\T_\nub L_\one}\to L_\nub\to M_\nub\to 0,\\
0\to Q\labelto{\T_\nub} Q\to M_\nub\to 0.
\end{gather*}
Since all $\nu_i$ are prime to $p$, we have $|\Gamma|\cdot M_\nub=p^m M_\nub=M_\nub$,
and by \cite[Corollary of Proposition 3]{Roiter}
the morphisms of $\Z[\Gamma]$-modules  $L_\nub\to M_\nub$ and  $Q\to M_\nub$
are projective in the sense of \cite[\S\,1]{Roiter}.
Now by \cite[Proposition 2]{Roiter}
there exists an isomorphism of $\Gamma$-lattices $L_\nub\oplus Q\simeq
{\T_\nub L_\one}\oplus Q$. Since $Q$ is a quasi-permutation $W$-lattice, it is
a quasi-permutation $\Gamma$-lattice, and by Lemma \ref{lem:qp-eq}
below, $L_\nub\sim {\T_\nub L_\one}$ as $\Gamma$-lattices. Since ${\T_\nub L_\one}\simeq
L_\one$, we conclude that $L_\nub\sim L_\one$.
\qed
\end{subsec}

\begin{lemma}\label{lem:quotient}
With the above notation
$L_\nub/{\T_\nub L_\one}\simeq Q/\T_\nub Q=\bigoplus_{i=2}^r Q_i/\nu_i Q_i$.
\end{lemma}
\begin{proof}
We have ${\T_\nub L_\one}=\langle S_\nub\rangle$, where $S_\nub=\{\nu_i\alpha_{\vk,i}\}_{i,\vk}\cup\{w_\nub\}$.
Note that
$$
dw_\nub=\sum_{i=1}^r \nu_i[(n_i-1)\alpha_{1,i}+(n_i-2)\alpha_{2,i}+\dots +2\alpha_{n_i-2,i}+\alpha_{n_i-1,i}].
$$
We see that $dw_\nub$ is a linear combination with integer coefficients  of $\nu_i\alpha_{\vk,i}$
and that $\alpha_{n_1-1,1}$ appears in this linear combination with coefficient 1
(because $\nu_1=1$).
Set $B'_\nub=S_\nub\smallsetminus\{\alpha_{n_1-1,1}\}$, then $\langle B'_\nub\rangle\ni \alpha_{n_1-1,1}$,
hence $\langle B'_\nub\rangle=\langle S_\nub\rangle={\T_\nub L_\one}$,
thus $B'_\nub$ is a basis of ${\T_\nub L_\one}$.
Similarly, the set $B_\nub:=\{\alpha_{\vk,i}\}_{i,\vk}\cup\{w_\nub\}\smallsetminus \{\alpha_{n_1-1,1}\}$ is a basis of $L_\nub$.
Both bases $B_\nub$ and $B'_\nub$ contain $\alpha_{1,1},\dots,\alpha_{n_1-2,1}$ and $w_\nub$.
For all $i=2,\dots,r$ and all $\vk=1,\dots,n_i-1$, the basis $B_\nub$ contains $\alpha_{\vk,i}$,
while $B'_\nub$ contains  $\nu_i\alpha_{\vk,i}$.
We see that the homomorphism of $W$-modules $Q/\T_\nub Q = \bigoplus_{i=2}^r Q_i/\nu_i Q_i \to L_\nub/{\T_\nub L_\one}$ is an isomorphism.
\end{proof}

\begin{lemma}\label{lem:qp-eq}
Let $\Gamma$ be a finite group, $A$ and $A'$ be $\Gamma$-lattices.
If $A\oplus B\sim A'\oplus B'$, where $B$ and $B'$ are quasi-permutation $\Gamma$-lattices,
then $A\sim A'$.
\end{lemma}
\begin{proof}
Since $B$ and $B'$ are quasi-permutation, by Lemma \ref{lem:quasi-p}
they are equivalent to 0, and we have
$$
A=A\oplus 0\sim A\oplus B\sim A'\oplus B'\sim A'\oplus 0=A'.
$$
This completes the proof of Lemma \ref{lem:qp-eq} and hence of Lemma
\ref{lem:equivalent}.
\end{proof}

To complete the proof of Proposition  \ref{prop:latticeL} it suffices to prove the next lemma.

\begin{lemma}\label{lem:not-qp}
Let $p| d$ be a prime.
Then there exists a subgroup $\Gamma\subset W$ isomorphic to $(\Z/p\Z)^m$ for some natural $m$
such that the $\Gamma$-lattice  $L_\one:=L_{(1,\dots,1)}$ is not quasi-permutation.
\end{lemma}

\begin{subsec}
Denote by $U_i$ the space $\R^{n_i}$ with canonical basis $\ve_{1,i}\, ,\, \ve_{2,i}\,,\,\dots,\,\ve_{n_i,i}$\,.
Denote by $V_i$ the subspace of codimension 1 in $U_i$ consisting of vectors with zero sum of the coordinates.
The group $W_i=\SS_{n_i}$ (the symmetric group) permutes the basis vectors
$\ve_{1,i}\, ,\, \ve_{2,i}\,,\,\dots,\,\ve_{n_i,i}$ and thus acts on $U_i$ and $V_i$.
Consider the homomorphism of vector spaces
$$
\chi_i\colon U_i\to\R,\quad \sum_{\lambda = 1}^{n_i}\beta_{\lambda,i}\ve_{\lambda,i}\mapsto  \sum_{\lambda = 1}^{n_i}\beta_{\lambda,i}
$$
taking a vector to the sum of its coordinates.
Clearly this homomorphism is $W_i$-equivariant, where $W_i$ acts trivially on $\R$.
We have  short exact sequences
$$
0\to V_i\to U_i\labelto{\chi_i}\R\to 0.
$$
Set $U=\bigoplus_{i=1}^r U_i,\ V=\bigoplus_{i=1}^r V_i$. The group $W=\prod_{i=1}^r W_i$ naturally acts on $U$ and $V$,
and we have an exact sequence of $W$-spaces
\begin{equation}\label{eq:exact1}
0\to V\to U\labelto{\chi}\R^r\to 0,
\end{equation}
where $\chi=(\chi_i)_{i=1,\dots,r}$ and $W$ acts trivially on $\R^r$.
\sms

Set $n=\sum_{i=1}^r n_i$.
Consider the vector space  $\Ubar:=\R^n$ with canonical basis $\veb_1,\veb_2,\dots,\veb_n$.
Consider the natural isomorphism
$$\varphi\colon U=\bigoplus_i U_i\isoto\Ubar$$
that takes $\ve_{1,1},\ve_{2,1},\dots,\ve_{n_1,1}$ to $\veb_1,\veb_2,\dots,\veb_{n_1}$,
takes  $\ve_{1,2},\ve_{2,2},\dots,\ve_{n_2,2}$ to $\veb_{n_1+1},\veb_{n_1+2},\dots,\veb_{n_1+n_2}$, and so on.
Let $\Vbar$ denote the subspace of codimension 1 in $\Ubar$ consisting of vectors with zero sum of the coordinates.
Sequence \eqref{eq:exact1} induces an exact sequence of $W$-spaces
\begin{equation}\label{eq-exact0}
0\to \varphi(V)\to \Vbar\labelto{\psi} \R^r\labelto{\Sigma} \R\to 0.
\end{equation}
Here $\psi=(\psi_i)_{i=1,\dots,r}$, where $\psi_i$ takes a vector
$\sum_{j=1}^n \beta_{j}\,\veb_{j}\in \Vbar$ to $\sum_{\lambda=1}^{n_i} \beta_{n_1+\dots+n_{i-1}+\lambda}$, and the map
$\Sigma$ takes a vector in $\R^r$ to the sum of its coordinates.
Note that $W$ acts trivially on $\R^r$ and $\R $.
\sms

We have a lattice $Q_i\subset V_i$ for each $i=1,\dots, r$, a lattice $Q=\bigoplus_i Q_i\subset \bigoplus_i V_i$,
and a lattice $\Qbar:=\Z \AA_{n-1}$ in $\Vbar$ with basis $\veb_1-\veb_2,\dots,\veb_{n-1}-\veb_n$.
The isomorphism  $\varphi$ induces an embedding of $Q=\bigoplus_i Q_i$ into $\Qbar$.
Under this embedding
\begin{align*}
&\alpha_{1,1}\mapsto \alphabar_1,\ \alpha_{2,1}\mapsto \alphabar_2,\ \dots,\ \alpha_{n_1-1,1}\mapsto\alphabar_{n_1-1},\\
&\alpha_{1,2}\mapsto \alphabar_{n_1+1},\ \alpha_{2,2}\mapsto \alphabar_{n_1+2},\ \dots,\ \alpha_{n_2-1,2}\mapsto\alphabar_{n_1+n_2-1},\\
&\dots\dots\dots\dots\dots\dots\dots\dots\dots\dots\dots\dots\dots\dots\dots\dots\dots\\
&\alpha_{1,r}\mapsto \alphabar_{n_1+n_2+\dots+n_{r-1}+1},\ \dots,\ \alpha_{n_r-1,r}\mapsto\alphabar_{n-1},
\end{align*}
while $\alphabar_{n_1},\alphabar_{n_1+n_2},\ \dots,\ \alphabar_{n_1+n_2+\dots+n_{r-1}}$ are skipped.
\end{subsec}

\begin{subsec}
We write $L$ for $L_\one$ and $w$ for $w_\one\in \frac{1}{d} Q$, where $Q=\bigoplus_i Q_i$. Then
$$
w=\sum_{i=1}^r w_i, \quad w_i=\frac{1}{d}[(n_i-1)\alpha_{1,i}+\dots+\alpha_{n_i-1,i}].
$$
Recall that
$$
Q_i=\Z \AA_{n_i-1}=\{(a_j)\in \Z^{n_i}\mid \sum_{j=1}^{n_i} a_j=0\}.
$$

Set
$$
\wbar=\frac{1}{d}\sum_{j=1}^{n-1} (n-j)\alphabar_j.
$$
Set $\Lambda_n(d)=\langle \Qbar,\wbar\rangle$.
Note that $\Lambda_n(d)=Q_n(n/d)$ with the notation of \cite[Subsection 6.1]{LPR}.
Set
$$
N=\varphi(Q\otimes_\Z\R)\cap\Lambda_n(d)=\varphi(V)\cap\Lambda_n(d).
$$
\end{subsec}

\begin{lemma}\label{lem:varphi}
$\varphi(L)=N$.
\end{lemma}

\begin{proof}
Write $j_1=n_1,\ j_2=n_1+n_2,\ \dots,\ j_{r-1}=n_1+\dots+n_{r-1}$.
Set $J=\{1,2,\dots,n-1\}\smallsetminus \{j_1,j_2,\dots,j_{r-1}\}$.
Set
$$
\mu=\frac{1}{d} \sum_{j\in J}  (n-j)\alphabar_j=\wbar-\sum_{i=1}^{r-1} \frac{n-j_i}{d}\alphabar_{j_i}.
$$
Note that $d|n$ and $d|j_i$ for all $i$, hence
the coefficients $(n-j_i)/d$ are integral, and therefore $\mu\in\Lambda_n(d)$.
Since also $\mu\in\varphi(Q\otimes_\Z \R)$, we see that $\mu\in N$.
\sms

Let $y\in N$.
Then
$$
y=b\wbar+\sum_{j=1}^{n-1} a_j\alphabar_j
$$
where $b,a_j\in\Z$, because $y\in\Lambda_n(d)$.
We see that in the basis $\alphabar_1,\dots,\alphabar_{n-1}$ of $\Lambda_n(d)\otimes_\Z \R$,
the element $y$ contains $\alphabar_{j_i}$ with coefficient
$$
b\frac{n-j_i}{d} +a_{j_i}.
$$
Since $y\in \varphi( Q\otimes_\Z \R)$, this coefficient must be 0:
$$
b\frac{n-j_i}{d}+ a_{j_i}=0.
$$
Consider
\begin{align*}
y-b\mu=y-b\left(\wbar-\sum_{i=1}^{r-1}\frac{n-j_i}{d}\alphabar_{j_i}\right)=
&y-b\wbar+\sum_{i=1}^{r-1}\frac{b(n-j_i)}{d}\alphabar_{j_i}\\
=\sum_{j=1}^{n-1} &a_j\alphabar_j+\sum_{i=1}^{r-1}\frac{b(n-j_i)}{d}\alphabar_{j_i}=\sum_{j\in J} a_j\alphabar_j\,,
\end{align*}
where $a_j\in\Z$.
We see that $y\in \langle \alphabar_j\ (j\in J),\mu\rangle$ for any $y\in N$,
hence $N\subset \langle \alphabar_j\ (j\in J),\mu\rangle$.
Conversely, $\mu\in N$  and $\alphabar_j \in N$ for $j\in J$,
hence $\langle \alphabar_j\ (j\in J),\mu\rangle\subset N$, thus
\begin{equation}\label{eq:N-mu}
N=\langle \alphabar_j\ (j\in J),\mu\rangle.
\end{equation}
\sms

Now
$$
\varphi(w)=\frac{1}{d}\left[\sum_{j=1}^{n_1-1} (n_1-j)\alphabar_j +\sum_{j=1}^{n_2-1} (n_2-j)\alphabar_{n_1+j}+\dots +\sum_{j=1}^{n_r-1} (n_r-j)\alphabar_{j_{r-1}+j}\right]
$$
while
$$
\mu=\frac{1}{d}\left[\sum_{j=1}^{n_1-1} (n-j)\alphabar_j +\sum_{j=1}^{n_2-1} (n-n_1-j)\alphabar_{n_1+j}+\dots +\sum_{j=1}^{n_r-1} (n_r-j)\alphabar_{j_{r-1}+j}\right].
$$
Thus
$$
\mu=\varphi(w)+\frac{n-n_1}{d}\sum_{j=1}^{n_1-1}\alphabar_j +\frac{n-n_1-n_2}{d}\sum_{j=1}^{n_2-1}\alphabar_{n_1+j}+\dots+\frac{n_r}{d}\sum_{j=1}^{n_r-1}\alphabar_{j_{r-1}+j}\,,
$$
where the coefficients
$$
\frac{n-n_1}{d},\quad \frac{n-n_1-n_2}{d},\quad\dots,\quad\frac{n_r}{d}
$$
are integral.
We see that
\begin{equation}\label{eq:mu-varphi(w)}
\langle\alphabar_j\ (j\in J),\ \mu\rangle=\langle\alphabar_j\ (j\in J),\ \varphi(w)\rangle.
\end{equation}
From \eqref{eq:N-mu} and \eqref{eq:mu-varphi(w)} we obtain that
\[
N=\langle\alphabar_j\, (j\in J),\,\mu\rangle=\langle\alphabar_j\, (j\in J),\, \varphi(w)\rangle=\varphi(L).\qedhere
\]
\end{proof}

\begin{subsec}
Now let $p|\gcd(n_1,\dots,n_r)$.
Recall that $W=\prod_{i=1}^r \SS_{n_i}$.
Since $p|n_i$ for all $i$, we can naturally embed $(\SS_p)^{n_i/p}$ into $\SS_{n_i}$.
We obtain a natural embedding
$$
\Gamma:=(\Z/p\Z)^{n/p}\into (\SS_p)^{n/p}\into W.
$$

In order to prove Lemma \ref{lem:not-qp}, it suffices to prove the next Lemma \ref{lem:p-p^2}.
Indeed, if $n$ has an odd prime factor $p$, then by Lemma \ref{lem:p-p^2} $L$ is not quasi-permutation.
If $n=2^s$, then we take $p=2$.
By the assumptions of Proposition \ref{prop:latticeL}, $n>4=2^2$,
and again by Lemma \ref{lem:p-p^2} $L$ is not quasi-permutation.
This proves  Lemma \ref{lem:not-qp}.
\end{subsec}

\begin{lemma}\label{lem:p-p^2}
If either $p$ odd or $n>p^2$, then $L$ is not quasi-permutation as a $\Gamma$-lattice.
\end{lemma}

\begin{proof}
By Lemma \ref{lem:varphi} it suffices to show that $N$ is not quasi-permutation.
Since $N=\Lambda_n(d)\cap\varphi(V)$,
we have an embedding
$$
\Lambda_n(d)/N\into \Vbar/\varphi(V).
$$
By \eqref{eq-exact0} $\Vbar/\varphi(V)\simeq\R^{r-1}$
and $W$ acts on $\Vbar/\varphi(V)$ trivially.
Thus $\Lambda_n(d)/N\simeq \Z^{r-1}$ and $W$ acts on $\Z^{r-1}$ trivially.
We have an exact sequence of $W$-lattices
$$
0\to N\to \Lambda_n(d)\to \Z^{r-1}\to 0,
$$
with trivial action of $W$ on $\Z^{r-1}$. We obtain that
$N\sim\Lambda_n(d)$ as a $W$-lattice, and hence, as a
$\Gamma$-lattice. Therefore, it suffices to show that
$\Lambda_n(d)=Q_n(n/d)$ is not quasi-permutation as a
$\Gamma$-lattice if either $p$ is odd or $n>p^2$. This is done in
\cite{LPR} in the proofs of Propositions 7.4 and 7.8. This completes
the proof of Lemma \ref{lem:p-p^2} and hence those of Lemma \ref{lem:not-qp} and
Proposition  \ref{prop:latticeL}.
\end{proof}

\section{Quasi-permutation lattices -- case $\AA_{n-1}$}
\label{sec:4}

In this section we prove Theorem \ref{thm:product} in the special case
when $G$ is isogenous to a direct product of groups of type $\AA_{n-1}$ for $n\ge 3$.
\sms

We maintain the notation of Subsection \ref{subsec:qp-An-nu}.
Let $L$ be an intermediate lattice between $Q$ and $P$, i.e., $Q\subset L\subset P$ (\,$L=Q$ are $L=P$ are possible).
Let $S$ denote the image of $L$ in $F$, then $L$ is the preimage of $S\subset F$ in $P$.
Since $W$ acts trivially on $F$, the subgroup $S\subset F$ is $W$-invariant, and therefore,
the sublattice $L\subset P$ is $W$-invariant.

\begin{theorem}\label{thm:qp-lattices}
With the notation of Subsection \ref{subsec:qp-An-nu} assume that $n_i\ge 3$ for all $i=1,2,\dots,r$.
Let $L$ between $Q$ and $P$ be an intermediate lattice,
and assume that $L\cap P_i=Q_i$ for all $i$ such that $n_i=3$ or $n_i=4$.
If $L$ is a quasi-permutation $W$-lattice, then $L=Q$.
\end{theorem}

\begin{proof}
We prove the theorem by induction on $r$.
The case $r=1$ follows from our assumptions if $n_1=3$ or $n_1=4$, and from Proposition \ref{prop:CK,LPR} if $n_1> 4$.
\sms

We assume that $r>1$ and that the assertion is true for $r-1$. We prove it for $r$.
\sms

For $i$ between $1$ and $r$ we set
$$
Q'_i=\bigoplus_{j\neq i} Q_j,\quad P'_i=\bigoplus_{j\neq i} P_j,\quad F'_i=\bigoplus_{j\neq i} F_j,\quad  W_i'=\prod_{j\neq i} W_j,
$$
then $Q'_i\subset Q$, $P'_i\subset P$, $F'_i\subset F$ and $W'_i\subset W$.
If $L$ is a quasi-permutation $W$-lattice, then by Lemma \ref{lem:2.8} $L\cap P'_i$ is  a quasi-permutation $W'_i$-lattice,
and by the induction hypothesis $L\cap P'_i= Q'_i$.
\sms

Now let $Q\subset L\subset P$, and assume that $L\cap P'_i= Q'_i$ for all $i=1,\dots, r$.
We shall show that if $L\neq Q$ then $L$ is not a quasi-permutation $W$-lattice.
This will prove Theorem \ref{thm:qp-lattices}.
\sms

 Assume that $L\neq Q$. Set $S=L/Q\subset F$, then $S\neq 0$.
We first show that $(L\cap P'_i)/ Q'_i=S\cap F'_i$.
Indeed, clearly $(L\cap P'_i)/ Q'_i\subset L/Q\cap P'_i/Q'_i= S\cap F'_i$.
Conversely, let $f\in S\cap F'_i$, then $f$ can be represented
by some $l\in L$ and by some $p\in P'_i$, and $q:=l-p\in Q$.
Since $L\supset Q$, we see that $p=l-q\in L\cap P'_i$, hence $f\in (L\cap P'_i)/Q'_i$,
and therefore $S\cap F'_i\subset (L\cap P'_i)/ Q'_i$.
Thus  $(L\cap P'_i)/ Q'_i=S\cap F'_i$.
\sms

By assumption we have $L\cap P'_i= Q'_i$, and we obtain that $S\cap F'_i=0$ for all $i=1,\dots, r$.
Let $S_{(i)}$ denote the image of $S$ under the projection $F\to F_i$.
We have a canonical epimorphism $p_i\colon S\to S_{(i)}$ with kernel $S\cap F'_i$.
Since  $S\cap F'_i=0$, we see that $p_i\colon S\to S_{(i)}$ is an isomorphism.
Set $q_i=p_i\circ p_1^{-1}\colon S_{(1)}\to S_{(i)}$, it is an isomorphism.
\sms

We regard $Q_i=\Z \AA_{n_i-1}$ and $P_i=\Lambda_{n_i}$ as the lattices described in \cite[Planche I]{Bourbaki}.
Then we have an isomorphism $F_i\cong \Z/n_i\Z$.
Since $S_{(i)}$ is a subgroup of the cyclic group $F_i\cong \Z/n_i\Z$ and $S\cong S_{(i)}$, we see that $S$ is a cyclic group,
and we see also that $|S|$ divides $n_i$ for all $i$, hence $d:=|S|$ divides $c:=\gcd(n_1,\dots,n_r)$.
\sms

We describe all  subgroups $S$ of order $d$ in $\bigoplus_{i=1}^r \Z/n_i\Z$
such that $S\cap (\bigoplus_{j\neq i} \Z/n_j\Z)=0$ for all $i$.
The element $a_i:=n_i/d+n_i\Z$ is a generator of $S_{(i)}\subset F_i=\Z/n_i\Z$.
Set $b_i=q_i(a_1)$.
Since $b_i$ is a generator of  $S_{(i)}$, we have $b_i=\nubar_i a_i$ for some $\nubar_i\in (\Z/d\Z)^\times$.
Let $\nu_i\in\Z$ be a representative of $\nubar_i$ such that $1\le \nu_i<d$, then $\gcd(\nu_i,d)=1$.
Moreover, since $q_1=\id$, we have $b_1=a_1$, hence $\nubar_1=1$ and $\nu_1=1$.
We obtain an element $\nub=(\nu_1,\dots,\nu_r)$.
With the notation of Subsection \ref{subsec:qp-An-nu} we have $S=S_\nub$ and $L=L_\nub$.
\sms

By Proposition \ref{prop:latticeL} $L_\nub$ is not a quasi-permutation $W$-lattice.
Thus $L$ is not quasi-permutation,
which completes the proof of Theorem \ref{thm:qp-lattices}.
\end{proof}

\section{Proof of Theorem \ref{thm:product}}
\label{sec:5}

\begin{lemma}[well-known]\label{lem:split}
Let $P_1$ and $P_2$ be abelian groups.
Set $P=P_1\oplus P_2=P_1\times P_2$, and let $\pi_1\colon P\to P_1$ denote the canonical projection.
Let $L\subset P$ be a subgroup.
If $\pi_1(L)=L\cap P_1$, then
$$ L=(L\cap P_1)\oplus (L\cap P_2).$$
\end{lemma}

\begin{proof}
Let $x\in L$. Set $x_1=\pi_1(x)\in\pi_1(L)$.
Since $\pi_1(L)=L\cap P_1$, we have $x_1\in L\cap P_1$.
Set $x_2=x-x_1$, then $x_2\in L\cap P_2$.
We have $x=x_1+x_2$.
This completes the proof of Lemma \ref{lem:split}.
\end{proof}

\begin{subsec}\label{subsec:almost}
Let $I$ be a finite set. For any $i\in I$ let $D_i$ be a connected Dynkin diagram.
Let $D=\bigcupdot_i D_i$ (disjoint union).
Let $Q_i$ and $P_i$ be the root and weight lattices of $D_i$, respectively, and $W_i$ be the Weyl group of $D_i$.
Set
$$
Q=\bigoplus_{i\in I} Q_i,\quad  P=\bigoplus_{i\in I} P_i, \quad W=\prod_{i\in I} W_i.
$$
\end{subsec}

\begin{subsec}
We construct certain quasi-permutation lattices $L$ such that $Q\subset L\subset P$.
\sms

Let $\{\, \{i_1,j_1\},\dots,\{i_s,j_s\}\, \}$ be a set of
non-ordered pairs in $I$ such that  $D_{i_l}$ and $D_{j_l}$ for all
$l=1,\dots,s$ are of type $\BB_1=\AA_1$ and all the indices
$i_1,j_1,\dots,i_s,j_s$ are distinct.
Fix such an $l$.
 We write
$\{i,j\}$ for $\{i_l,j_l\}$ and we set $D_{i,j}:=D_{i}\cupdot D_{j}$,
$Q_{i,j}:=Q_{i}\oplus Q_{j}$,  $P_{i,j}:=P_{i}\oplus P_{j}$. We
regard $D_{i,j}$ as a Dynkin diagram of type $\DD_2$, and we denote
by $M_{i,j}$ the intermediate lattice between $Q_{i,j}$ and
$P_{i,j}$ isomorphic to $\X(\SO_4)$, the character lattice of the
group $\SO_4$; see Section \ref{sec:1}, after Lemma \ref{lem:2.8}.
Let $f_i$ be  a generator of the lattice $Q_i$ of rank 1,  and let $f_j$ be a generator of $Q_j$,
then $P_i=\langle\half f_i\rangle$ and $P_j=\langle\half f_j\rangle$.
Set $e_1^{(l)}=\half(f_i+f_j)$, \ $e_2^{(l)}=\half(f_i-f_j)$,
then $\{e_1^{(l)},e_2^{(l)}\}$ is a basis of $M_{i,j}$\,, and
\begin{equation}\label{eq:M}
M_{i,j}=\left\langle Q_{i,j},e_1^{(l)}\right\rangle,\qquad P_{i,j}=\left\langle M_{i,j},\half(e_1^{(l)}+e_2^{(l)})\,\right\rangle.
\end{equation}
We have $M_{i,j}\cap P_i=Q_i,\ M_{i,j}\cap P_j=Q_j$, and $[M_{i,j}:Q_{i,j}]=2$.
Concerning the Weyl group, we have
$$W(D_{i,j})=W(D_i)\times W(D_j)=W(\DD_2)=\SS_2\times\{\pm 1\},$$
where the symmetric group $\SS_2$ permutes the basis vectors  $e_1^{(l)}$ and $e_2^{(l)}$ of $M_{i,j}$\,,
while the group $\{\pm 1\}$ acts on $M_{i,j}$  by multiplication by scalars.
 We say that $M_{i,j}$ is an {\em indecomposable quasi-permutation lattice}
(it corresponds to the semisimple Cayley group $\SO_4$
which does not decompose into a direct product of its  normal subgroups).

\sms

Set $I'=I\smallsetminus\bigcup_{l=1}^s \{i_l,j_l\}$. For $i\in I'$ let $M_i$ be any quasi-permutation intermediate lattice between $Q_i$ and $P_i$
(such an intermediate lattice exists if and only if $D_i$ is of one of the types
$\AA_n,\ \BB_n,\ \CC_n,\ \DD_n,\ \GG_2$, see Proposition \ref{prop:CK,LPR}).
We say that $M_i$ is a {\em simple quasi-permutation lattice} (it corresponds to a stably Cayley simple group).
We set
\begin{equation}\label{eq:*}
L= \bigoplus_{l=1}^s M_{i_l,j_l}\, \oplus\, \bigoplus_{i\in I'} M_i.
\end{equation}
We say that a lattice $L$ as in \eqref{eq:*} is a
{\em direct sum of indecomposable quasi-permutation  lattices  and simple quasi-permutation  lattices}.
Clearly $L$ is a quasi-permutation $W$-lattice.
\end{subsec}

\begin{theorem}\label{thm:almost}
Let $D, Q, P, W$ be as in Subsection \ref{subsec:almost}.
Let $L$ be an intermediate lattice between $Q$ and $P$, i.e., $Q\subset L\subset P$
(where $L=Q$ and $L=P$ are possible).
If $L$ is a quasi-permutation $W$-lattice, then $L$ is as in \eqref{eq:*}.
Namely, then $L$ is a direct sum of indecomposable  quasi-permutation  lattices $M_{i,j}$
for some set of pairs  $\{\, \{i_1,j_1\},\dots,\{i_s,j_s\}\, \}$
and some family of simple quasi-permutation intermediate lattices $M_i$
between $Q_i$ and $P_i$ for $i\in I'$.
\end{theorem}

\begin{remark}\label{rem:pairs}
The set of pairs $\{\, \{i_1,j_1\},\dots,\{i_s,j_s\}\, \}$ in Theorem \ref{thm:almost} is uniquely determined by $L$.
Namely, a pair $\{i,j\}$ belongs to this set if and only if the Dynkin diagrams $D_i$ and $D_j$ are of type $\BB_1=\AA_1$ and
$$
 L\cap P_i= Q_i\,,\quad L\cap P_j= Q_j\,,  \quad \text{while } L\cap (P_i\oplus P_j)\neq Q_i\oplus Q_j\,.
$$
\end{remark}

\begin{proof}[Proof of Theorem \ref{thm:almost}]
We prove the theorem by induction on $m=|I|$, where $I$ is as in Subsection \ref{subsec:almost}.
The case $m=1$ is trivial.
\sms

We assume that $m\ge 2$ and that the theorem is proved for all $m'<m$. We prove it for $m$.
First we consider three special cases.
\sms

{\em Split case}.
Assume that for some subset $A\subset I$, $A\neq I$, $A\neq\emptyset$,
we have $\pi_A(L)=L\cap P_A$, where $P_A=\bigoplus_{i\in A}P_i$ and $\pi_A\colon P\to P_A$ is the canonical projection.
Then by Lemma \ref{lem:split} we have $L=(L\cap P_A)\oplus(L\cap P_{A'})$, where $A'=I\smallsetminus A$.
By Lemma \ref{lem:2.8} $L\cap P_A$ is a quasi-permutation $W_A$-lattice, where $W_A=\prod_{i\in A} W_i$.
By the induction hypothesis  the lattice $L\cap P_A$ is  a direct sum
of indecomposable  quasi-permutation  lattices  and simple quasi-permutation  lattices.
Similarly, $L\cap P_{A'}$ is such a direct sum.
We conclude that $L=(L\cap P_A)\oplus(L\cap P_{A'})$ is such a direct sum, and we are done.

\sms

{\em $\AA_{n-1}$-case}.
Assume that all $D_i$ are of type $\AA_{n_i-1}$, where $n_i\ge 3$ (so $\AA_1$ is not permitted).
We assume also that when $n_i=3$ and when $n_i=4$ (that is, for $\AA_2$ and for $\AA_3=\DD_3$) we have $L\cap P_i=Q_i$
(for $n_i > 4$ this is automatic because $L\cap P_i$ is a quasi-permutation $W_i$-lattice, see Proposition \ref{prop:CK,LPR}).
In this case by Theorem \ref{thm:qp-lattices} we have $L=Q=\bigoplus Q_i$,
hence $L$ is a direct sum of simple quasi-permutation  lattices, and we are done.

\sms

{\em $\AA_1$-case}.
Assume that all $D_i$ are of type $\AA_1$.
Then by \cite[Theorem 18.1]{BKLR} the lattice $L$ is  a direct sum of indecomposable  quasi-permutation lattices
and simple quasi-permutation lattices, and we are done.
\sms

Now we shall show that these three special cases exhaust all the quasi-permutation lattices.
In other words, we shall show that if $Q\subset L\subset P$ and $L$ is not as in one of these three cases,
then $L$ is not quasi-permutation. This will complete the proof of the theorem.
\sms

For  the sake of contradiction, let us assume that $Q\subset L\subset P$,
that $L$ is not in one of the three special cases above,
and that $L$ is a {quasi-permutation} $W$-lattice.
We shall show in three steps that $L$ is as in Proposition \ref{prop:one-vector}.
By Proposition \ref{prop:one-vector}, $L$ is not quasi--permutation, which contradicts our assumptions.
This contradiction will prove the theorem.
\sms

{\em Step 1.}
 For $i\in I$ consider the intersection $L\cap P_i$, it is a quasi-permutation $W_i$-lattice (by Lemma \ref{lem:2.8}),
hence $D_i$ is of one of the types $\AA_{n-1}$, $\BB_n$, $\CC_n$, $\DD_n$, $\GG_2$ (by Proposition \ref{prop:CK,LPR}).
Note that $\pi_i(L)\neq L\cap P_i$ (otherwise we are in the split case).
\sms

Now assume that for some $i\in I$, the Dynkin diagram $D_i$ is of type $\GG_2$ or $\CC_n$ for some $n\ge 3$,
or  $D_i$ is of type $\AA_2$ and $L\cap P_i\neq Q_i$.
Then $L\cap P_i$ is a quasi-permutation $W_i$-lattice
(by Lemma \ref{lem:2.8}), hence $L\cap P_i=P_i$ (by  Proposition \ref{prop:CK,LPR}).
Since $P_i\supset \pi_i(L)\supset L\cap P_i$,
we obtain that $\pi_i(L)=L\cap P_i$, which is impossible.
Thus no $D_i$ can be of  type $\GG_2$ or $\CC_n$, $n\ge 3$,
and if $D_i$ is of type $\AA_2$ for some $i$, then $L\cap P_i=Q_i$.
\sms

Thus all $D_i$ are of types $\AA_{n-1},\ \BB_n$ or $\DD_n$,
and if $D_i$ is of type $\AA_2$ for some $i\in I$, then $L\cap P_i=Q_i$.
Since $L$ is not as in the $\AA_{n-1}$-case,
we may assume that one of the $D_i$, say $D_1$,
is of type $\BB_n$  for some $n\ge 1$ ($\BB_1=\AA_1$ is permitted),
or of type $\DD_n$ for some $n\ge 4$, or of type $\DD_3$ with $L\cap P_1\neq Q_1$.
Indeed, otherwise {\em all} $D_i$ are of type $\AA_{n_i-1}$ for $n_i\ge 3$,
and in the cases $\AA_2$ ($n_i=3$) and $\AA_3$ ($n_i=4$) we have $L\cap P_i=Q_i$,
i.e., we are in the $\AA_{n-1}$-case, which contradicts our assumptions.

\sms

{\em Step 2.}
In this step, using the Dynkin diagram $D_1$ of type  $\BB_n$ or $\DD_n$ from the previous step,
we construct a quasi-permutation sublattice $L'\subset L$ of index 2
such that $L'$ is as  in \eqref{eq:*}.
First we consider the cases $\BB_n$ and $\DD_n$ separately.
\sms

Assume that  $D_1$ is of type $\BB_n$  for some $n\ge 1$ ($\BB_1=\AA_1$ is permitted).
We have $[P_1:Q_1]=2$.
Since  $P_1\supset\pi_1(L)\supsetneq L\cap P_1\supset Q_1$,
we see that $\pi_1(L)=P_1$ and $L\cap P_1= Q_1$.
Set $M_1=Q_1$.
We have $\pi_1(L)=P_1$, $L\cap P_1=M_1$, and $[P_1:M_1]=2$.
\sms

Now assume that $D_1$ is of type $\DD_n$ for some $n\ge 4$,
or of type $\DD_3$ with $L\cap P_1\neq Q_1$.
Set $M_1=L\cap P_1$, then $M_1$ is a quasi-permutation $W_1$-lattice by Lemma  \ref{lem:2.8},
and it follows from  Proposition \ref{prop:CK,LPR}  that
$(W_1,M_1)\simeq \X(\SO_{2n})$,
where $\X(\SO_{2n})$ denotes the character lattice of $\SO_{2n}$;
see Section \ref{sec:1}, after Lemma \ref{lem:2.8}.
It follows that  $[M_1:Q_1]=2$ and $[P_1:M_1]=2$.
Since $P_1\supset\pi_1(L)\supsetneq L\cap P_1= M_1$,
we see that $\pi_1(L)=P_1$.
Again we have $\pi_1(L)=P_1$, $L\cap P_1=M_1$, and $[P_1:M_1]=2$.
\sms

Now we consider together the cases when $D_1$
is of type $\BB_n$  for some $n\ge 1$
and when $D_1$ is of type $\DD_n$ for some $n\ge 3$, where in the case $\DD_3$ we have $L\cap P_1\neq Q_1$.
Set
$$
L':=\ker[L\labelto{\pi_1} P_1\to P_1/M_1].
$$
Since $\pi_1(L)=P_1$, and $[P_1:M_1]=2$, we have $[L:L']=2$.
Clearly we have   $\pi_1(L')=M_1$.
Set
$$
L^\dagger_1:=\ker[\pi_1\colon L\to P_1]=L\cap P'_1,
$$
where $P'_1=\bigoplus_{i\neq 1}P_i$.
Since $L$ is a quasi-permutation $W$-lattice,
by Lemma \ref{lem:2.8} the lattice $L_1^\dagger$ is a quasi-permutation $W_1'$-lattice,
where $W'_1=\prod_{i\neq 1}W_i$.
By the induction hypothesis, $L_1^\dagger$ is
a direct sum of indecomposable  quasi-permutation  lattices
and simple quasi-permutation  lattices as in \eqref{eq:*}.
Since $M_1=L\cap P_1$, we have $M_1\subset L'\cap P_1$,
and $L'\cap P_1\subset L\cap P_1=M_1$, hence $L'\cap P_1=M_1=\pi_1(L')$,
and by Lemma  \ref{lem:split} we have $L'=M_1\oplus L^\dagger_1$.
Since $M_1$ is a simple quasi-permutation lattice,
we conclude that $L'$ is a direct sum of indecomposable  quasi-permutation  lattices
and simple quasi-permutation  lattices  as in \eqref{eq:*},
and $[L:L']=2$.
\sms

{\em Step 3.}
In this step we show that $L$ is as in Proposition  \ref{prop:one-vector}.
We write
\begin{equation*}\label{eq:sum-L'}
L'= \bigoplus_{l=1}^s (L'\cap P_{i_l,j_l}) \, \oplus\, \bigoplus_{i\in I'} (L'\cap P_i)  ,
\end{equation*}
where $P_{i_l,j_l}=P_{i_l}\oplus P_{j_l}$,
the Dynkin diagrams $D_{i_l}$ and $D_{j_l}$ are of type $\AA_1=\BB_1$, and
$L'\cap P_{i_l,j_l}=M_{i_l,j_l}$ as in \eqref{eq:M}.
For any $i\in I'$,
we have $[\pi_i(L):\pi_i(L')]\le 2$, because $[L:L']=2$.
Furthermore, for $i\in I'$ we have
$$
\pi_i(L')=L'\cap P_i\subset L\cap P_i\subsetneq \pi_i(L),
$$
hence
$[\pi_i(L): (L\cap P_i)]=2$
and $L'\cap P_i=L\cap P_i$.
Similarly, for any $l=1,\dots, s$, if we write $i=i_l,\ j=j_l$, then we have
$$
M_{i,j}=L'\cap P_{i,j}\subset L\cap P_{i,j}\subsetneq \pi_{i,j}(L)\subset P_{i,j},\qquad [P_{i,j}:M_{i,j}]=2,
$$
whence $\pi_{i,j}(L)=P_{i,j}$, $L\cap P_{i,j}=M_{i,j}$, and therefore $[\pi_{i,j}(L):(L\cap P_{i,j})]=\quad[P_{i,j}:M_{i,j}]=2$
and $L'\cap P_{i,j}=M_{i,j}=L\cap P_{i,j}$.
\sms

We view the Dynkin diagram $D_{i_l}\cupdot D_{j_l}$ of type $\AA_1\cupdot\AA_1$ corresponding
to the pair $\{i_l,j_l\}$ ($l=1,\dots,s$) as a Dynkin diagram of type $\DD_2$.
Thus we view $L'$  as a direct sum of indecomposable  quasi-permutation  lattices  and simple quasi-permutation  lattices
corresponding to Dynkin diagrams of type $\BB_{n}$, $\DD_{n}$ and $\AA_{n}$.
\sms

We wish to show that $L$ is as in Proposition \ref{prop:one-vector}.
We change our notation in order to make it closer to that of
Proposition \ref{prop:one-vector}.
\sms

As in Subsection \ref{subsec:one-vector}, we now write $D_i$ for Dynkin diagrams
of types $\BB_{l_i}$ and $\DD_{l_i}$ only, appearing in $L'$,
where $\BB_1=\AA_1$, $\BB_2=\CC_2$, $\DD_2=\AA_1\cupdot \AA_1$ and $\DD_3=\AA_3$ are permitted,
but for $\DD_{l_i}$ with $l_i=2,3$  we require that
$$L\cap P_i=M_i:=\X(\SO_{2l_i}).$$
We write $L'_i:=L\cap P_i=L'\cap P_i$.
We have $[\pi_i(L):L'_i]=2$, hence $[P_i:L'_i]\ge 2$.
If $D_i$ is of type $\BB_{l_i}$, then $[P_i:L'_i]=2$.
If $D_i$ is of type $\DD_{l_i}$, then $L'_i=L\cap P_i\neq Q_i$,
for $\DD_2$ and $\DD_3$ by our assumption
and for  $\DD_{l_i}$ with $l_i\ge 4$ because $L\cap P_i$ is a quasi-permutation $W_i$-lattice (see  Proposition \ref{prop:CK,LPR});
again we have $[P_i:L'_i]=2$.
We see that for all $i$ we have $[P_i:L'_i]=2$, $\pi_i(L)=P_i$,
and the lattice $L'_i=M_i$ is as in Subsection \ref{subsec:one-vector}.
We realize the root system $R(D_i)$ of type $\BB_{l_i}$ or $\DD_{l_i}$
in the standard way (cf. \cite[Planches II, IV]{Bourbaki})
in the space $V_i:=\R ^{l_i}$ with basis $(e_s)_{s\in S_i}$, then
$L'_i$ is the lattice generated by the basis vectors $(e_s)_{s\in S_i}$ of $V_i$,
and we have $P_i=\langle L'_i, \half x_i\rangle$, where
$$
x_i=\sum_{s\in S_i} e_s\in L'_i\,.
$$
In particular, when $D_i$ is of type $\DD_2$  we have $x_i=e_1^{(l)}+e_2^{(l)}$
with the notation of formula \eqref{eq:M}.
\sms

As in Subsection \ref{subsec:one-vector}, we write $\Delta_\iota$
for Dynkin diagrams of type $\AA_{n_\iota'-1}$ appearing in $L'$,
where $n_\iota'\ge 3$ and for $\AA_3=\DD_3$ we require that $L\cap P_\iota=Q_\iota$.
We write $L'_\iota:=L\cap P_\iota=L'\cap P_\iota$.
Then $L'_\iota=Q_\iota$ for all $\iota$,
for $\AA_2$ by Step 1, for $\AA_3$ by our assumption,
and for other $\AA_{n_\iota'-1}$ because $L'_\iota$ is a
quasi-permutation $W_\iota$-lattice; see Proposition \ref{prop:CK,LPR}.
We have $\pi_\iota(L)\supsetneq L\cap P_\iota= L'_\iota$
and $[\pi_\iota(L): L'_\iota]=[\pi_\iota(L): \pi_\iota(L')]\le 2$ (because $[L:L']=2$).
It follows that $[\pi_\iota(L):L'_\iota]=2$, i.e., $[\pi_\iota(L):Q_\iota]=2$.
We know that $P_\iota/Q_\iota$ is a
cyclic group of order $n'_\iota$. Since it has a subgroup
$\pi_\iota(L)/Q_\iota$ of order 2, we conclude that $n'_\iota$ is
even, $n'_\iota=2n_\iota$ (where $2n_\iota\ge 4$), and
$\pi_\iota(L)/Q_\iota$ is the unique subgroup of order 2 of the
cyclic group $P_\iota/Q_\iota$ of order $2n_\iota$.
As in Subsection \ref{subsec:one-vector},
we realize
the root system $\Delta_\iota$ of type $\AA_{2n_\iota-1}$
in the standard way (cf. \cite[Planche I]{Bourbaki})
in  the subspace $V_\iota$ of vectors with zero sum of the coordinates
in  the space $\R^{2n_\iota}$ with basis
$\ve_{\iota,1},\dots, \ve_{\iota,2n_\iota}$.
We set
$$
\xi_\iota=\ve_{\iota,1}-\ve_{\iota,2}+\ve_{\iota,3}-\ve_{\iota,4}+\dots+\ve_{\iota,2n_\iota-1}-\ve_{\iota,2n_\iota},
$$
then $\xi_\iota\in L'_\iota$ and  $\half\xi_\iota\in\pi_\iota(L)\smallsetminus L'_\iota$
(cf. \cite[Planche I, formula (VI)\,]{Bourbaki}\,),
hence $\pi_\iota(L)=\langle L'_\iota,\half\xi_\iota\rangle$.
\sms

Now we set
$$
v=\half\sum_{i\in I} x_i+\half\sum_{\iota=1}^\mu \xi_\iota\,.
$$
We claim that
$$
L=\langle L',v\rangle.
$$
\sms

\noindent
{\em Proof of the claim.}
Let $w\in L\smallsetminus L'$, then $L=\langle L',w\rangle$, because $[L:L']=2$.
Set $z_i=\half x_i-\pi_i(w)$, then $z_i\in L'_i\subset L'$, because $\half x_i,\pi_i(w)\in \pi_i(L)\smallsetminus L'_i$.
Similarly, we set $\zeta_\iota=\half\xi_\iota-\pi_\iota(w)$, then $\zeta_\iota\in L'_\iota\subset L'$.
We see that
$$
v=w+\sum_i z_i+\sum_\iota\zeta_\iota,
$$
where $\sum_i z_i+\sum_\iota\zeta_\iota\in L'$, and the claim follows.
\qed
\sms

It follows from the claim that $L$ is as in  Proposition
\ref{prop:one-vector} (we use the assumption that we are not in the
$\AA_1$-case). Now by Proposition  \ref{prop:one-vector} $L$ is not
quasi-invertible, hence not quasi-permutation, which contradicts our
assumptions. This contradiction proves Theorem \ref{thm:almost}.
\end{proof}

\begin{proof}[Proof of Theorem \ref{thm:product}]
Theorem \ref{thm:product} follows immediately from Theorem \ref{thm:almost} by virtue of Proposition \ref{prop:LPR,BKLR}.
\end{proof}

\section{Proof of Theorem \ref{thm:product-non-closed}}
\label{sec:6}

In this section we deduce  Theorem \ref{thm:product-non-closed} from Theorem  \ref{thm:product}.
\sms

Let $G$ be a stably Cayley semisimple $k$-group.
Then $\Gbar :=G\times_k \kbar$ is stably Cayley over an algebraic closure $\kbar$ of $k$.
By Theorem \ref{thm:product},
 $G_\kbar=\prod_{j\in J} G_{j,\kbar}$ for some finite index set $J$,
where each $G_{j,\kbar}$ is either a stably Cayley simple group or is
isomorphic to $\gSO_{4,\kbar}$.
(Recall that $\gSO_{4,\kbar}$ is stably Cayley and semisimple,
but is not simple.)
Here we write $G_{j,\kbar}$ for the factors in order to emphasize
that they are defined over $\kbar$.
By Remark \ref{rem:pairs} the collection of direct factors $G_{j,\kbar}$
is determined uniquely by $\Gbar$.
 The Galois group $\Gal(\kbar/k)$ acts on $G_\kbar$, hence on $J$.
Let $\Omega$ denote the set of orbits of $\Gal(\kbar/k)$ in $J$.
For $\omega\in\Omega$ set $G^\omega_\kbar=\prod_{j\in\omega} G_{j,\kbar}$\,,
then $\Gbar=\prod_{\omega\in \Omega} G^\omega_\kbar$.
Each $G^\omega_\kbar$ is $\Gal(\kbar/k)$-invariant,
hence it defines a $k$-form $G^\omega_k$ of $G^\omega_\kbar$.
We have $G=\prod_{\omega\in \Omega} G^\omega_k$.
\sms

For each $\omega\in\Omega$  choose $j=j_\omega\in\omega$.
Let $l_j/k$ denote the Galois extension in $\kbar$
corresponding to the stabilizer of $j$ in $\Gal(\kbar/k)$.
The subgroup $G_{j,\kbar}$ is $\Gal(\kbar/l_j)$-invariant,
hence it comes from  an $l_j$-form $G_{j,l_j}$.
By the definition of Weil's restriction of scalars (see e.g. \cite[Subsection 3.12]{Voskresenskii-book})
$G^\omega_k\cong R_{l_j/k} G_{j,l_j}$,
hence $G\cong\prod_{\omega\in \Omega}  R_{l_j/k} G_{j,l_j}$.
Each $G_{j,l_j}$ is either absolutely simple or an $l_j$-form of $\gSO_4$.
\sms

We complete the proof using an argument from \cite[Proof of Lemma 11.1]{BKLR}.
We  show that $G_{j,l_j}$ is a direct factor of $G_{l_j}:=G\times_k l_j$.
It is clear from the definition that $G_{j,\kbar}$ is a direct factor
of $G_\kbar$ with complement $G'_\kbar=\prod_{i\in J\smallsetminus \{j\} } G_{i,\kbar}$.
Then $G'_\kbar$ is $\Gal(\kbar/l_j)$-invariant, hence it comes from some $l_j$-group $G'_{l_j}$.
We have $G_{l_j}=G_{j,l_j}\times_{l_j} G'_{l_j}$, hence $G_{j,l_j}$ is a direct factor of $G_{l_j}$.
\sms

Recall that  $G_{j,l_j}$ is either a form of $\gSO_4$ or absolutely simple.
If it  is a form of $\gSO_4$, then  clearly it is stably Cayley over $l_j$.
It remains to show that if  $G_{j,l_j}$ is absolutely simple,
then $G_{j,l_j}$ is stably Cayley over $l_j$.
The group $G_\kbar$ is  stably Cayley over $\kbar$.
Since $G_{j,\kbar}$ is a direct factor of the stably Cayley $\kbar$-group $G_\kbar$
over the algebraically closed field $\kbar$,
by \cite[Lemma 4.7]{LPR} $G_{j,\kbar}$ is stably Cayley over $\kbar$.
Comparing \cite[Theorem 1.28]{LPR} and \cite[Theorem 1.4]{BKLR},
we see that $G_{j,l_j}$ is either stably Cayley over $l_j$ (in which case
we are done) or an outer form of $\gPGL_{2n}$ for some  $n\ge 2$.
Thus assume by way of contradiction that $G_{j,l_j}$
is an outer form of $\gPGL_{2n}$ for some  $n \ge 2$.
Then by \cite[Example 10.7]{BKLR} the character lattice of $G_{j,l_j}$ is not quasi-invertible,
and by \cite[Proposition 10.8]{BKLR} the group  $G_{j,l_j}$
cannot be a direct factor of a stably Cayley $l_j$-group.
This contradicts  the fact that $G_{j,l_j}$ is a direct factor of the stably Cayley $l_j$-group $G_{l_j}$.
We conclude that $G_{j,l_j}$ cannot be an outer form of $\gPGL_{2n}$ for any  $n \ge 2$.
Thus $G_{j,l_j}$ is stably Cayley over $l_j$, as desired.
\qed
\bigskip

\appendix
\section{Appendix: Some quasi-permutation character lattices}\label{app}

The positive assertion of Proposition \ref{prop:CK,LPR} above is well known.
It is contained in \cite[Theorem 0.1]{CK} and in \cite[Theorem 1.4]{BKLR}.
However, \cite{BKLR} refers to \cite[Theorem 0.1]{CK},
and \cite{CK} refers to a series of results on rationality
(rather than only stable rationality) of corresponding generic tori.
In this appendix for the reader's convenience
we provide a proof of the following positive result
in terms of lattices only.

\begin{proposition} \label{prop:positive}
 Let  $G$ be any form of one of the following groups
\[
\text{$\gSL_3$, $\gPGL_{n}$ ($n$ odd), $\gSO_n$ ($n\ge 3$),
$\Sympl_{2n}$, $\gG_2$}
\]
or an {\em inner} form of $\gPGL_{n}$ ($n$ even).
Then the character lattice of $G$ is quasi-permutation.
\end{proposition}

\begin{proof}
{\em $\gSO_{2n+1}$.}
Let $L$ be the character lattice of $\gSO_{2n+1}$ (including $\gSO_3$).
Then the Dynkin diagram is $D=\BB_n$.
The Weyl group is $W=\SS_n\ltimes (\Z/2\Z)^n$.
Then $L=\Z^n$ with the standard basis $e_1,\dots,e_n$.
The group $\SS_n$ naturally permutes $e_1, \ldots, e_n$,
while $(\Z/2\Z)^n$ acts by sign changes.
Since $W$ permutes the basis up to $\pm$ sign,
the $W$-lattice $L$ is quasi-permutation,
see \cite[\S\,2.8]{Lorenz}.
\sms

{\em $\gSO_{2n}$, any form, inner or outer.}
Let $L$ be the character lattice of $\gSO_{2n}$ (including $\gSO_4$).
Then the Dynkin diagram is $D=\DD_n$, with root system $R=R(D)$.
We consider the pair $(A,L)$ where $A=\Aut(R,L)$, then $(A,L)$
is isomorphic to the character lattice of $\gSO_{2n+1}$, hence is quasi-permutation.
\sms

{\em $\gSp_{2n}$.} The character lattice of $\gSp_{2n}$ is isomorphic to
the character lattice of $\gSO_{2n+1}$, hence is quasi-permutation.
\sms

{\em $\gPGL_n$, inner form.}
The character lattice of $\gPGL_n$ is the root lattice $L=Q$ of $\AA_{n-1}$.
It is a quasi-permutation $\SS_n$-lattice,
cf. \cite[Example 2.8.1]{Lorenz}.

\sms
{\em $\gPGL_n$, outer form, $n$ odd.}
Let $P$ be the {\em weight} lattice of $\AA_{n-1}$, where $n\ge 3$ is odd.
Then $P$ is generated by elements $e_1,\dots,e_n$ subject to the relation
$$
e_1+\dots+e_n=0.
$$
The automorphism group $A=\Aut(\AA_{n-1})$ is the product of $\SS_n$ and $\SS_2$.
The group $A$ acts on $P$ as follows: $\SS_n$ permutes $e_1,\dots,e_n$,
and the nontrivial element of $\SS_2$ takes each $e_i$ to $-e_i$.
\sms

We denote by $M$ the $A$-lattice of rank $2n+1$ with basis $s_1,\dots,s_n, t_1,\dots,t_n, u$.
The group $\SS_n$ permutes $s_i$ and permutes $t_i$ $(i=1,\dots,n)$,
and the nontrivial element of $\SS_2$ permutes $s_i$ and $t_i$ for each $i$.
The group $A$ acts trivially on $u$.
Clearly $M$ is a permutation lattice.
\sms

We define an $A$-epimorphism $\pi\colon M\to P$ as follows:
$$
\pi\colon\quad s_i\mapsto e_i,\quad t_i\mapsto -e_i,\quad u\mapsto 0.
$$
Set $M'=\ker \pi$, it is an $A$-lattice of rank $n+2$. We show that it is a permutation lattice.
We write down a set of $n+3$ generators  of $M'$:
$$
\rho_i=s_i+t_i,\quad \sigma=s_1+\dots+s_n,\quad \tau=t_1+\dots+t_n,\quad u.
$$
There is a relation
$$
\rho_1+\dots+\rho_n=\sigma+\tau.
$$
We define a new set of $n+2$ generators:
$$
\rhotil_i=\rho_i+u,\quad \sigmatil=\sigma+\tfrac{n-1}{2}u,\quad \tautil=\tau+\tfrac{n-1}{2}u,
$$
where $\tfrac{n-1}{2}$ is integral because $n$ is odd.
We have
$$
\rhotil_1+\dots+\rhotil_n-\sigmatil-\tautil=u,
$$
hence this new set indeed generates $M'$, hence it is a basis.
The group $\SS_n$ permutes $\rhotil_1,\dots,\rhotil_n$, while $\SS_2$ permutes $\sigmatil$ and $\tautil$.
Thus $A$ permutes our basis, and therefore $M'$ is a permutation lattice.
We have constructed a left resolution of $P$:
$$
0\to M'\to M\to P\to 0,
$$
(with permutation lattices $M$ and $M'$),
which by duality gives a right resolution
of the {\em root} lattice $Q\cong P^\vee$ of $\AA_{n-1}$:
$$
0\to Q\to M^\vee\to (M')^\vee \to 0
$$
with permutation lattices $M^\vee$ and $(M')^\vee$.
Thus the character lattice $Q$ of $\gPGL_n$
is a quasi-permutation $A$-lattice for odd $n$.
\sms

The assertion that the character lattice of $G$ is quasi-permutation in the remaining cases $\gSL_3$ and $\gG_2$
follows from the next Lemma~\ref{lem:Voskresenskii}.

\begin{lemma}[{\cite[Lemma 2.5]{BKLR}}]\label{lem:Voskresenskii}
Let $\Gamma$ be a finite group and $L$ be any $\Gamma$-lattice
of rank $r=1$ or $2$.  Then $L$ is quasi-permutation.
\end{lemma}

This lemma, which is a version of~\cite[\S\,4.9, Examples 6 and 7]{Voskresenskii-book},
was stated in \cite{BKLR} without proof.
For the sake of completeness we supply a short proof here.
\sms

We may assume that $\Gamma$ is a maximal finite subgroup of $\gGL_r(\bbZ)$.
If $r=1$, then $\gGL_1(\bbZ) = \{ \pm 1 \}$,
and the lemma reduces to the case of the character
lattice of $\gSO_3$ treated above.
\sms

Now let $r=2$.  Up to conjugation there are two maximal finite subgroups
of $\GL_2(\bbZ)$, they are isomorphic to  the dihedral groups $D_8$ (of order 8)
and to $D_{12}$ (of order  $12$), resp.,
see e.g. ~\cite[\S\,1.10.1, Table 1.2]{Lorenz}.
The group $D_8$ is the group of symmetries of a square,
and in this case it suffices to show that
the character lattice of $\gSO_5$ is quasi-permutation,
which we have done above.  The group $D_{12}$ is the group
of symmetries of a regular hexagon, and  in this case
it suffices to show that the character lattice of $\gPGL_3$ (outer form)
is quasi-permutation, which we have done above as well.
This completes the proofs of Lemma~\ref{lem:Voskresenskii} and
Proposition~\ref{prop:positive}.
\end{proof}

\sms

\noindent{\sc Acknowledgements.}
The authors are very grateful to the anonymous referee
for prompt and thorough refereeing the paper
and for noticing a (correctable) error in Theorem 4.1 and in the proof of Theorem 5.4.
The authors thank Rony A.~Bitan for his help in proving Lemma \ref{lem:not-qp}.
The first-named author was supported in part
by the Hermann Minkowski Center for Geometry.
The second-named author was supported in part by the Israel Science
Foundation, grant 1207/12, and by the Minerva Foundation through the
Emmy Noether Institute for Mathematics.

\sms

\Addresses

\end{document}